\title{\large{\textbf{LEVEL SET PERCOLATION FOR RANDOM INTERLACEMENTS AND THE GAUSSIAN FREE FIELD}}}
\date{}

\documentclass[a4paper,11pt, leqno]{article}
\usepackage[english]{babel}
\usepackage[a4paper,left=4.0cm,right=4.0cm,top=3.0cm,bottom=2.8cm]{geometry}
\usepackage{amsfonts, amsmath, amssymb, amsthm, mathrsfs}
\usepackage{stmaryrd}
\usepackage{lmodern}

\usepackage{psfrag}
\usepackage{graphicx}
\usepackage{verbatim}
\usepackage{verbatim}
\usepackage[font=small]{caption}

\usepackage[hang,flushmargin]{footmisc}

\numberwithin{equation}{section}

\newtheorem{theorem}{Theorem}[section]
\newtheorem{lemma}[theorem]{Lemma}

\newtheorem{proposition}[theorem]{Proposition}

\newtheorem{corollary}[theorem]{Corollary}

\theoremstyle{remark}
\newtheorem*{proof1}{Proof of Lemma \ref{L:U,ALPHA_BAD}}
\newtheorem*{proof2}{Proof of Lemma \ref{L:FE2_AbsCont}}

\theoremstyle{definition}
\newtheorem{remark}[theorem]{Remark}

\addtocounter{section}{-1}

\begin{document}

\maketitle
\thispagestyle{empty}

\begin{center}
\vspace{-1.1cm}
Pierre-Fran\c cois Rodriguez\footnote{\noindent Departement Mathematik, ETH Z\"urich, CH-8092 Z\"urich, Switzerland. \\ \indent    This research was supported in part by the grant ERC-2009-AdG 245728-RWPERCRI.} \\
\end{center}
\vspace{1.0cm}
\begin{abstract}
\centering
\begin{minipage}{0.9\textwidth}
We consider continuous-time random interlacements on $\mathbb{Z}^d$, $d\geq3$, and investigate the percolation model where a site $x$ of $\mathbb{Z}^d$ is occupied if the total amount of time spent at $x$ by all the trajectories of the interlacement at level $u \geq 0$ exceeds some constant $\alpha \geq 0$, and empty otherwise. We also investigate percolation properties of empty sites. A recent isomorphism theorem \cite{S2} enables us to ``translate'' some of the relevant questions into the language of level-set percolation for the Gaussian free field on $\mathbb{Z}^d$, $d\geq3$, about which new insights of independent interest are also gained.
\end{minipage}
\end{abstract}

\vspace{8cm}
\begin{flushright}
December 2013
\end{flushright}

\newpage

\thispagestyle{empty}
\mbox{}
\newpage

\setcounter{page}{1}

\section{Introduction} \label{INTRODUCTION}

In the present work, we consider the field of occupation times for continuous-time random interlacement at level $u \geq 0$ on $\mathbb{Z}^d$, $d \geq 3$, and investigate the percolative properties of the random subset of $\mathbb{Z}^d$ obtained by keeping only those sites at which the occupation time exceeds some given cut-off value $\alpha \geq 0$. We also consider the percolative properties of the complement of this set in $\mathbb{Z}^d$. Our main interest is to infer for which values of the parameters $(u,\alpha)$ these random sets percolate. A recent isomorphism theorem \cite{S2} relates the field of occupation times for continuous-time random interlacements on $\mathbb{Z}^d$, $d \geq 3$ (and more generally, on any transient weighted graph) to the Gaussian free field on the same graph. We will exploit this correspondence as a transfer mechanism to reformulate some of the problems in terms of questions regarding level-set percolation for the Gaussian free field. This will allow us to use certain renormalization techniques recently developed in this context in \cite{RS}. Additionally, we derive new results concerning ``two-sided'' level-set percolation for the Gaussian free field on $\mathbb{Z}^d$, $d\geq 3$, where, in contrast to $(0.2)$ of \cite{RS} (see also \cite{BLM}), the level sets consist of those sites at which the \textit{absolute value} of the corresponding field variable exceeds a certain level $h\geq 0$. 

\bigskip

We now describe our results and refer to Section \ref{NOTATION} for details.  We consider continuous-time random interlacements on $\mathbb{Z}^d$, $d\geq 3$. Somewhat informally, this model can be defined as a cloud of simple random walk trajectories modulo time-shift on $\mathbb{Z}^d$ constituting a Poisson point process, where a non-negative parameter $u$ appearing multiplicatively in the intensity measure regulates how many paths enter the picture (we defer a precise definition to the next section, see the discussion around \eqref{intensity}). For any $u \geq 0$ and $\alpha \geq 0$, we introduce the (random) subsets of $\mathbb{Z}^d$ 
\begin{equation} \label{I^u,alpha}
\mathcal{I}^{u,\alpha} = \{ x\in \mathbb{Z}^d  \;; L_{x,u} > \alpha \}, \qquad \mathcal{V}^{u,\alpha} = \{ x\in \mathbb{Z}^d  \;; L_{x,u} \leq  \alpha \} = \mathbb{Z}^d \setminus \mathcal{I}^{u,\alpha}, 
\end{equation}
where $(L_{x,u})_{x \in \mathbb{Z}^d}$ denotes the field of occupation times at level $u$, see \eqref{L}, and ask for which values of the parameters $u$ and $\alpha$ these sets percolate. Note that for all $u \geq 0$, $\mathcal{I}^{u,0}$  corresponds to the (discrete-time) interlacement set at level $u$ introduced in $(0.7)$ of \cite{S1} (see also \eqref{I^u} and \eqref{I^u,0} below) and $\mathcal{V}^{u,0}$ to the according vacant set. Before addressing the core issue of describing the phase diagrams for percolation of the random sets $\mathcal{I}^{u,\alpha}$ and $\mathcal{V}^{u,\alpha}$, as $u$ and $\alpha$ vary, we prove uniqueness of the infinite clusters, whenever they exist. More precisely, we show in Corollary \ref{C:UNIQUENESS} that for all $u \geq 0$, $\alpha > 0$ and $d\geq3$,
\begin{equation} \label{RESULT:UNIQUENESS}
\text{$\mathbb{P}$-a.s., $\mathcal{I}^{u,\alpha}$ and $\mathcal{V}^{u,\alpha}$ contain \textit{at most one} infinite connected component},
\end{equation}
where $\mathbb{P}$ denotes the law of the interlacement point process, as defined below \eqref{intensity}. For $\alpha =0 $, \eqref{RESULT:UNIQUENESS} is already known and follows from \cite{S1}, Corollary 2.3, and \cite{T1}, Theorem 1.1. 

\bigskip

Our main results concern the existence/absence of infinite clusters inside $\mathcal{I}^{u,\alpha}$ and $\mathcal{V}^{u,\alpha}$, in terms of the parameters $u$ and $\alpha$. Let us define the functions
\begin{equation} \label{eta^I}
 \eta^{\mathcal{I}}(u,\alpha) = \mathbb{P}\big[ 0 \stackrel{\mathcal{I}^{u,\alpha}}{\longleftrightarrow} \infty \big], \qquad  \eta^{\mathcal{V}}(u,\alpha) = \mathbb{P}\big[ 0 \stackrel{\mathcal{V}^{u,\alpha}}{\longleftrightarrow} \infty \big], \qquad \text{for $u \geq 0$, $\alpha \geq 0$},
\end{equation}
to denote the probabilities that $0$ lies in an infinite cluster of $\mathcal{I}^{u,\alpha}$ and $\mathcal{V}^{u,\alpha}$, respectively. Observing that $ \eta^{\mathcal{I}}(u,\alpha) $ is decreasing in $\alpha$ for every (fixed) value of $u \geq 0$, it is sensible to introduce the critical parameter
\begin{equation} \label{alpha_*}
 \alpha_*(u) = \inf \{ \alpha \geq 0 \; ; \; \eta^{ \mathcal{I}}(u,\alpha) =0 \} \in [0,\infty] , \quad \text{for }u \geq 0
\end{equation} 
(with the convention $\inf \emptyset = \infty$). It is not difficult to see that the function $\alpha_*(\cdot)$ is non-decreasing, see \eqref{alpha_*_monotone} below. Our main results regarding percolation of the sets $\mathcal{I}^{u,\alpha}$ state that
\begin{equation} \label{RESULT: alpha_*}
0 < \alpha_*(u) < \infty, \qquad \text{for all $u > 0$ and $d \geq 3$}
\end{equation}
(see Theorem \ref{T:ALPHA_*PERC} for positivity of $\alpha_*(u)$ and Theorem \ref{T:ALPHA_*} for finiteness). In words, the sets $ \mathcal{I}^{u,\alpha}$ exhibit a non-trivial percolation phase transition as $\alpha$ varies, for every (fixed) positive value of $u$. 

In a similar vein, for $\mathcal{V}^{u,\alpha}$, we introduce the critical parameter
\begin{equation} \label{u_*}
u_*(\alpha) = \inf\{ u\geq 0 \; ; \; \eta^{\mathcal{V}}(u,\alpha) = 0  \} \in [0,\infty], \quad \text{for } \alpha \geq 0,
\end{equation}
which is well-defined since $\eta^{\mathcal{V}}(\cdot,\alpha)$ is non-increasing for every value of $\alpha \geq 0$ (we will comment on the asymmetry in the role of $u$ and $\alpha$ in \eqref{alpha_*} and \eqref{u_*} below; see the discussion following \eqref{isom_THM}). It is an easy matter to verify that the function $u_*(\cdot)$ is non-decreasing, see \eqref{u_*_monotone}, and that $u_*(0)=u_*$, where $u_*$ refers to the critical point for percolation of the vacant set of (discrete-time) random interlacements, as defined in $(0.13)$ of \cite{S1}, which is known to be finite and strictly positive for all dimensions $d \geq 3$, see \cite{S1}, Theorem 4.3, and \cite{SS2}, Theorem 3.4 (see also Theorem 5.1 in \cite{S3} for a more general result). Our main conclusion concerning percolation of the sets $\mathcal{V}^{u,\alpha}$, see Theorem \ref{T:U_*} below, asserts that
\begin{equation} \label{RESULT: u_*}
 (0 < u_* \leq) \ u_*(\alpha) < \infty, \qquad \text{for all $ \alpha \geq 0$ and $d \geq 3$}. 
\end{equation}
In fact, not only are we able to establish finiteness of the critical parameters in \eqref{RESULT: alpha_*} and \eqref{RESULT: u_*}, but also the stronger result that (see \eqref{T:ALPHA_*2}, \eqref{T:U_*2} and Remark \ref{R:U_*})
\begin{equation} \label{RESULT:conn_function}
 \begin{split}
  &\text{the connectivity functions $\mathbb{P}\big[ 0 \stackrel{\mathcal{I}^{u,\alpha}}{\longleftrightarrow} x \big]$ and $\mathbb{P} \big[ 0 \stackrel{\mathcal{V}^{u,\alpha}}{\longleftrightarrow} x \big]$ have stretched} \\
&\text{exponential decay in $x$ as $|x| \to \infty$, for all $u \geq 0 $ and $\alpha= \alpha(u)$ sufficiently large,} \\
&\text{respectively for all $\alpha \geq 0$ and $u=u(\alpha)$ sufficiently large,}
 \end{split}
\end{equation}
where the events in the probabilities refer to the existence of a nearest-neighbor path in $\mathcal{I}^{u,\alpha}$, resp. $\mathcal{V}^{u,\alpha}$, connecting $x$ to the origin. Tentative phase diagrams for percolation of $\mathcal{I}^{u, \alpha}$ and $\mathcal{V}^{u, \alpha}$, $u,  \alpha \geq 0$, can be found in Figure \ref{phasediagram} below.

\begin{figure} [h!] 
\centering
\psfragscanon
\includegraphics[width=16cm]{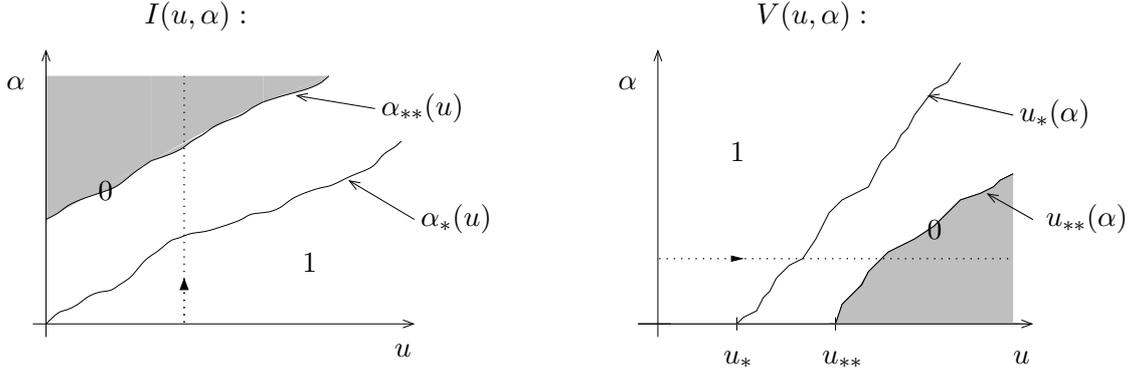}
\caption{the functions $I(u,\alpha)= 1\{\text{$\mathbb{P}$-a.s., $\mathcal{I}^{u,\alpha}$ has a unique infinite component} \}$ and $V(u,\alpha)= 1\{\text{$\mathbb{P}$-a.s., $\mathcal{V}^{u,\alpha}$ has a unique infinite component} \}$. The shaded areas, in which the corresponding connectivity functions have stretched exponential decay, define the auxiliary critical lines $\alpha_{**}(u)$ and $u_{**}(\alpha)$, see Remark \ref{R:U_*}, 2). It is presently an open problem whether the two critical lines $\alpha_*(\cdot)$ and $\alpha_{**}(\cdot)$, respectively
 $u_*(\cdot)$ and $u_{**}(\cdot)$, actually coincide.}
\label{phasediagram}
\end{figure}

As hinted above, some of the proofs rely on Theorem 0.1 of \cite{S2}, which relates $(L_{x,u})_{x\in \mathbb{Z}^d}$ to the Gaussian free field on $\mathbb{Z}^d$, see \eqref{isom_THM}. In particular, en route to proving \eqref{RESULT:conn_function}, we show the following result, interesting in its own right. Let $P^G$ denote the canonical law of Gaussian free field on $\mathbb{Z}^d$, i.e. $P^G$ is the probability measure on $\mathbb{R}^{\mathbb{Z}^d}$ such that,
\begin{equation} \label{GFF}
\begin{split}
&\text{under $P^G$, the canonical field $\varphi$ = $(\varphi_x)_{x \in \mathbb{Z}^d}$ is a centered Gaussian} \\
&\text{field with covariance $\mathbb{E}[\varphi_x \varphi_y] = g (x,y)$}, \text{ for all } x,y \in \mathbb{Z}^d,
\end{split}
\end{equation}
where $g(\cdot,\cdot)$ denotes the Green function of simple random walk on $\mathbb{Z}^d$, cf. \eqref{GreenFunction}. For arbitrary $h\geq 0$, we consider the ``two-sided'' level set
\begin{equation} \label{levelsets}
\mathcal{L}^{\geq h} = \{ x\in\mathbb{Z}^d \; ;\; | \varphi_x| \geq h \}.
\end{equation}
Introducing the critical parameter 
\begin{equation} \label{h_*}
\mathfrak{h}_* = \inf \big\{ h\geq 0 \; ; \; P^G[0\stackrel{\mathcal{L}^{\geq h}}{\longleftrightarrow} \infty ] = 0 \big\},
\end{equation}
we show in Theorem \ref{T:GFF_NOPERC} that
\begin{equation} \label{RESULT: h_*}
 \mathfrak{h}_* < \infty, \qquad \text{for all $d \geq 3$},
\end{equation}
and, similarly to \eqref{RESULT:conn_function}, that the connectivity function of $\mathcal{L}^{\geq h}$ has stretched exponential decay for sufficiently large $h$, see \eqref{T:GFF_NOPERC3}. This strengthens the result $(0.5)$ of \cite{RS} (see also \cite{BLM}), which states that $h_*$, the critical level for percolation of the (one-sided) level sets $\{ x \in \mathbb{Z}^d \; ; \; \varphi_x \geq h\}$, $h \in \mathbb{R}$, is finite for all $d \geq 3$. Moreover, it follows from Theorem 3.3 of \cite{RS} that $\mathfrak{h}_*$ is strictly positive in large dimensions, see Remark \ref{R:h_*positive}, 2) below. It was already known from Theorem 7 in \cite{G} (see also p. 281 therein) that there is no directed percolation inside $\mathcal{L}^{\geq h}$ when $h$ is sufficiently large, for all $d\geq 4$. Finally, let us mention that our results concerning $\mathcal{L}^{\geq h}$ might be helpful for investigating certain random conductance models on $\mathbb{Z}^d$, in the spirit of \cite{BD}, with nearest-neighbor conductances involving the Gaussian free field; see also \cite{DSZ, DS} for further motivation.

\bigskip

We now comment on the proofs. In order to establish the uniqueness result \eqref{RESULT:UNIQUENESS}, cf. Corollary \ref{C:UNIQUENESS} below, we invoke a classical theorem of Burton and Keane (see \cite{BK}, Theorem \nolinebreak 2) after showing in Theorem \ref{T:FE} that for all $u,\alpha>0$, the translation invariant law $Q_{u,\alpha}$ of $(1\{ x \in \mathcal{I}^{u,\alpha} \})_{x\in \mathbb{Z}^d}$ under $\mathbb{P}$, see \eqref{Q_u,alpha} and Lemma \ref{trans_inv}, has the so-called finite energy property, i.e.  
\begin{equation} \label{RESULT:FINITE_ENERGY}
0< Q_{u,\alpha}(Y_0 =1 | \sigma( Y_x, \; x \neq 0))<1, \text{ $Q_{u,\alpha}$-a.s. for all $u>0$, $\alpha >0$ and $d \geq 3$},
\end{equation}
where $Y_x$, $x \in \mathbb{Z}^d$, refer to the canonical coordinates on $\{0,1 \}^{\mathbb{Z}^d}$. This differs markedly from the case $\alpha = 0$, since the law of random interlacement at any level $u\geq 0$ fails to fulfill the finite-energy condition, see \cite{S1}, Remark 2.2, 3). The proof of the lower bound in \eqref{RESULT:FINITE_ENERGY} involves a delicate \textit{local} ``surgery operation'' on paths, which roughly consists of sending a ``furtive'' trajectory to $x$ which spends enough time there to ensure that $L_{x,u} > \alpha$ without spoiling a given configuration outside of $x$. The proof of the upper bound essentially requires us to prevent $x$ from being visited too much, which is a considerably simpler task.

The positivity of $\alpha_*(u)$ for $u>0$ in \eqref{RESULT: alpha_*}, cf. Theorem \ref{T:ALPHA_*PERC} below, is shown as follows. First, we introduce new occupation variables on $\mathbb{Z}^d$, whereby a site $x$ is ``occupied'' if and only if $x \in \mathcal{I}^{u,0}$ \textit{and} the first-passage holding time at $x$ of the trajectory in the interlacement cloud with smallest label ($ \leq u$) passing through $x$ exceeds $\alpha$ (see \eqref{T:U_*perc_pf1} for a precise definition). In particular, this implies that $L_{x,u}> \alpha$ whenever $x$ is ``occupied,'' hence $Q_{u,\alpha}$ dominates the (joint) law of these new occupation variables. Loosely speaking, we then prove that, conditionally on $\mathcal{I}^{u,0}$, these new variables define an independent Bernoulli percolation on the discrete interlacement set $\mathcal{I}^{u,0}$ with a suitable success parameter $p(\alpha)$ satisfying $ \lim_{\alpha \to 0} p(\alpha) = 1$. This enables us to use some recent results of \cite{RaSa} to infer that the set of occupied vertices has an infinite cluster if $p(\alpha)$ is sufficiently close to 1 (i.e. if $\alpha$ is small enough).

The proofs of the finiteness of $\alpha_*(u)$ and $u_*(\alpha)$ in \eqref{RESULT: alpha_*} and \eqref{RESULT: u_*}, see Theorems \ref{T:ALPHA_*} and \ref{T:U_*}, respectively, both rely on the aforementioned isomorphism theorem (see \cite{S2}, Theorem 0.1), which states that
\begin{equation} \label{isom_THM}
\begin{split}
&\big( L_{x,u}+ \frac{1}{2}\varphi_x^2 \big)_{x \in \mathbb{Z}^d}, \text{ under $\mathbb{P}  \otimes P^G$, has the} \\
&\text{same law as}\big( \frac{1}{2}(\varphi_x+\sqrt{2u})^2 \big)_{x \in \mathbb{Z}^d}, \text{ under $P^G$}.
\end{split}
\end{equation}
We focus on the claim $\alpha_*(u)< \infty$ first. Thus, we consider $\mathcal{I}^{u,\alpha}$ for fixed $u \geq 0$, as $\alpha$ becomes large. By \eqref{isom_THM}, $Q_{u,\alpha}$ (the law of $(1\{ L_{x,u} > \alpha \})_{x\in \mathbb{Z}^d}$ under $\mathbb{P}$) is dominated by the law of $\big(1\{ (\varphi_x+\sqrt{2u})^2/2 > \alpha \}\big)_{x\in \mathbb{Z}^d}$ under $P^G$. Hence, intuitively, if $L_{x,u}> \alpha$, then $|\varphi_x + \sqrt{2u}| > \sqrt{2 \alpha}$, i.e. $|\varphi_x|$ has to be large (since $\alpha$ is). This heuristic reasoning suggests that the asserted finiteness of $\alpha_*(u)$ is in fact a corollary of \eqref{RESULT: h_*}. 

The proof of \eqref{RESULT: u_*} is somewhat more involved, but has a similar flavor. Suppose that $x \in \mathcal{V}^{u,\alpha}$, i.e. $L_{x,u} \leq \alpha$, for some fixed $\alpha >0$. From the ``equality''  $ L_{x,u}+ \frac{1}{2}\varphi_x^2  \text{ ``$=$'' } \frac{1}{2}(\widetilde{\varphi}_x+\sqrt{2u})^2$, we deduce that, as $u \to \infty$, either $\widetilde{\varphi}$ is very negative (as to counteract the effect of $\sqrt{2u}$), or $|\varphi_x|$ must be large (since $L_{x,u}$ stays bounded), and both are rather unlikely by virtue of \eqref{RESULT: h_*}. This indicates that a subcritical phase for $\mathcal{V}^{u,\alpha}$ should emerge when $u$ becomes sufficiently large, cf. also Fig. \ref{phasediagram}. Note that the preceding discussion also accounts for the incongruent roles of $u$ and $\alpha$ in the definitions \eqref{alpha_*} and \eqref{u_*}, which is due to the way we apply \eqref{isom_THM}.

Finally, the result \eqref{RESULT: h_*} concerning ``two-sided'' level-set percolation for the Gaussian free field, see Theorem \ref{T:GFF_NOPERC} below, is shown using some of the tools developed in \cite{RS} for the analysis of (one-sided) level-set percolation (i.e. percolation of the sets $\{ x\in \mathbb{Z}^d \: ; \: \varphi_x \geq h \},$ $h \in \mathbb{R}$; the corresponding critical parameter is denoted by $h_*$, see (0.4) in \cite{RS}). In particular, it involves a renormalization scheme akin to the one introduced in Section 2 of \cite{RS} (see also \cite{SS}, \cite{S3}), and crucially depends on the decoupling inequality (Proposition 2.2 in \cite{RS}; see also Proposition \ref{P:DEC_INEQ} below) derived therein. However, we cannot simply follow the strategy used to prove finiteness of $h_* $ in \cite{RS} in order to establish \eqref{RESULT: h_*}, because the relevant crossing events $\{B(0,L) \stackrel{\mathcal{L}^{\geq h}}{\longleftrightarrow} S(0,2L) \}$, with $L \geq 1$,  $h \geq0$, which refer to the existence of a (nearest-neighbor) path in $\mathcal{L}^{\geq h}$ connecting $B(0,L)$, the closed ball of radius $L$ around the origin in the $\ell^\infty$-norm, to $S(0, 2L)$, the $\ell^\infty$-sphere of radius $2L$ around $0$, are neither increasing nor decreasing ``in $\varphi$,'' so Proposition 2.2 of \cite{RS} does not apply directly. To overcome this difficulty, we proceed as follows. First, we partition $\mathbb{Z}^d$ into disjoint boxes of equal side length $L_0$, for some $L_0 \geq 1$, and call any such box $h$-\textit{bad} if $|\varphi_x| >h$ for at least one site $x$ inside the box (this is quite crude but suffices for our purpose). Next, we consider the quantities
\begin{equation*}
\begin{split}
 q_n^+(h) \ \text{``$=$'' }  \ P^G [&\text{the box $B(0,L_n)$ contains $2^n$ ``well-separated'' boxes of side} \\
&\text{length $L_0$, each of which contains at least one site $x$ with $\varphi_x \geq h$}],
\end{split}
\end{equation*}
for $h \geq 0$, where $(L_n)_{n\geq 0}$ is a geometrically increasing sequence of length scales, see \eqref{L_n}. We define $q_n^-(h)$ similarly, with the last condition replaced by the requirement that $\varphi_x \leq - h$ for some site $x$ in the given box of side length $L_0$. Using the results of \cite{RS}, we show that $\lim_{n\to \infty} q_n^{\pm}(h) = 0$ for some careful choice of the parameters $h$ and $L_n$, $n \geq 0$. Together with a geometric argument in the spirit of Lemma 6 in \cite{RaSa}, see Lemmas \nolinebreak \ref{L:CRAMGEN} and \ref{L:CRAM} below, this yields that large connected components of $h$-bad blocks have small probability. The claim \eqref{RESULT: h_*} then easily follows, since the existence of an infinite cluster in $\mathcal{L}^{\geq h}$ implies the existence of an infinite connected component of $h$-bad boxes.

\medskip

We conclude this introduction by describing the organization of this article. In Section \nolinebreak \ref{NOTATION}, we introduce some basic notation, briefly review the definition of continuous-time random interlacements, and collect a few auxiliary properties of the measures $Q_{u,\alpha}$, $u,\alpha \geq 0$ (see above \eqref{RESULT:FINITE_ENERGY}). Sections \ref{FE} and \ref{I_PERC} are devoted to the uniqueness result \eqref{RESULT:UNIQUENESS}  (see Corollary \ref{C:UNIQUENESS}) and to the positivity of $\alpha_*$ in \eqref{RESULT: alpha_*} (see Theorem \ref{T:ALPHA_*PERC}), respectively. All results concerning absence of percolation are contained in Sections \ref{GFF_NOPERC} and \ref{RI_NOPERC}. Section \nolinebreak \ref{GFF_NOPERC} deals solely with the Gaussian free field, and  \eqref{RESULT: h_*} is shown in Theorem \ref{T:GFF_NOPERC}, after a suitable renormalization scheme has been set up. Section \ref{RI_NOPERC} addresses the question of absence of percolation for the sets $\mathcal{I}^{u,\alpha}$ and $\mathcal{V}^{u,\alpha}$. The main results \eqref{RESULT: alpha_*} and \eqref{RESULT: u_*} are established in Theorems \ref{T:ALPHA_*} and \ref{T:U_*}, respectively, along with the asserted decay behavior of the corresponding connectivity functions, see \eqref{RESULT:conn_function}.

\medskip

One final remark concerning our convention regarding constants: we denote by $c,c',\dots$ positive constants with values changing from place to place. Numbered constants $c_0,c_1,\dots$ are defined at the place they first occur within the text and remain fixed from then on until the end of the article. The dependence of constants (and other quantities) on the dimension $d$ of the lattice will be kept implicit throughout.

\section{Notation and useful results} \label{NOTATION}

In this section, we introduce some basic notation to be used in the sequel, recall the definition of continuous-time random interlacement on $\mathbb{Z}^d$, $d \geq 3$, and collect some auxiliary properties of the law $Q_{u,\alpha}$ of $\mathcal{I}^{u,\alpha}$ (see \eqref{Q_u,alpha} below), for $u,\alpha \geq 0$.

We denote by $\mathbb{N}= \{0,1,2,\dots\}$ the set of natural numbers, by $\mathbb{N}_* = \mathbb{N} \setminus \{ 0 \}$ the set of positive integers and by $\mathbb{Z}=\{ \dots ,-1,0,1,\dots \}$ the set of integers. We write $\mathbb{R}$ for the set of real numbers, $\mathbb{R}_+$ for the set of non-negative real numbers (this includes $0$), abbreviate $r \wedge s = \min \{r,s \}$ and $r \lor s = \max\{ r,s\}$ for any two numbers $r,s \in \mathbb{R}$, and $[r]$ for the integer part of $r$, for any $r \geq 0$. We consider the lattice $\mathbb{Z}^d$, and (tacitly) assume throughout that $d \geq 3$. On $\mathbb{Z}^d$, we respectively denote by $\vert \cdot \vert$ and $\vert \cdot \vert_\infty$ the Euclidean and $\ell^\infty$-norms. For any $x \in \mathbb{Z}^d$ and $r \geq 0$, we let $B(x,r) = \{ y \in \mathbb{Z}^d  \; ; \, \vert y-x \vert_\infty \leq r \}$ and $S(x,r) = \{ y \in \mathbb{Z}^d \; ; \; \vert y-x \vert_\infty = r \}$ stand for the $\ell^\infty$-ball and $\ell^\infty$-sphere of radius $r$ centered at $x$. Given $K$ and $U$ subsets of $\mathbb{Z}^d$, $K^c = \mathbb{Z}^d \setminus K$ stands 
for the complement of $K$ in $\mathbb{Z}^d$, $\vert K \vert$ for the cardinality of $K$, $K \subset \subset \mathbb{Z}^d $ means that $K \subset \mathbb{Z}^d$ and $\vert K \vert< \infty$, and $d(K,U) = \inf \{ \vert x-y \vert_\infty \; ; \; x \in K , y \in U\}$ denotes the $\ell^\infty$-distance between $K$ and $U$. If $K= \{ x\}$, we simply write $d(x, U)$. Finally, for any $K \subset \mathbb{Z}^d$, we define the inner boundary of $K$ to be the set $\partial^i K = \{ x \in K \; ; \; \exists y \in K^c, \vert y-x \vert =1 \}$, and the outer boundary of $K$ as $\partial K = \partial^i (K^c)$. 

We endow $\mathbb{Z}^d$ with the nearest-neighbor graph structure, i.e. the edge-set consists of all pairs of sites $\{ x ,y \}$, $x,y \in \mathbb{Z}^d$, such that $|x-y|=1$. We consider the spaces $\widehat{W}_+$, $\widehat{W}$ of infinite, respectively doubly infinite, $\mathbb{Z}^d \times (0, \infty)$-valued sequences, such that the $\mathbb{Z}^d$-valued sequences form an infinite, respectively doubly-infinite nearest-neighbor trajectory spending finite time in any finite subset of $\mathbb{Z}^d$, and such that the $(0, \infty)$-valued components have an infinite sum in the case of $\widehat{W}_+$, and infinite ``forward'' and ``backward'' sums, when restricted to positive and negative indices, in the case of $\widehat{W}$. We write $X_n$, $\sigma_n$ and $\theta_n$ with $n \geq 0$, respectively $n \in \mathbb{Z}$, for the $\mathbb{Z}^d$- and $(0,\infty)$-valued canonical coordinates and canonical shifts on $\widehat{W}_+$, respectively $\widehat{W}$, and denote by $\widehat{\mathcal{W}}_+$ and $\widehat{\mathcal{W}}$ the corresponding canonical $\sigma$-algebras. For $\widehat{w} \in \widehat{W}$, we will often abbreviate $X(\widehat{w})=(X_n(\widehat{w}))_{n \in \mathbb{Z}}$ and $\sigma(\widehat{w})=(\sigma_n(\widehat{w}))_{n \in \mathbb{Z}}$.

We let $P_x$, $x \in \mathbb{Z}^d$, be the law on $\widehat{W}_+$ under which $(X_n)_{n\geq 0}$ is distributed as simple random walk starting at $x$ and $\sigma_n$, $n\geq 0$, are i.i.d. exponential variables with parameter $1$, independent of the $X_n$, $n\geq 0$. Since $d \geq 3$, the walk is transient, so $\widehat{W}_+$ has full measure under $P_x$. We denote by $E_x$ the corresponding expectation. Moreover, for any measure $\rho$ on $\mathbb{Z}^d$, we write $P_\rho$ for the measure $\sum_{x\in \mathbb{Z}^d} \rho(x)P_x$, and $E_{\rho}$ for the corresponding expectation. We denote by $g(\cdot,\cdot)$ the Green function of simple random walk, i.e.
\begin{equation}\label{GreenFunction}
g(x,y) = \sum_{n \geq 0} P_x [X_n = y], \qquad \text{for } x,y \in \mathbb{Z}^d,
\end{equation}
which is finite (since $d \geq3$) and symmetric. Moreover, $g(x,y)= g(x-y,0) \stackrel{\text{def.}}{=} g(x-y)$ due to translation invariance. For $U \subseteq \mathbb{Z}^d$ and $\widehat{w}\in \widehat{W}_+$, we write $H_U(\widehat{w})$, $\widetilde{H}_U(\widehat{w})$ and $T_U(\widehat{w})$ for the entrance time in $ U $ , the hitting time of $U$ and the exit time from $U$ for the trajectory $\widehat{w}$, i.e.
\begin{equation}\label{stop_times}
\text{$H_U(\widehat{w}) = \inf \{ n\geq 0 \; ; \; X_n(\widehat{w}) \in U \}$, $\widetilde{H}_U(\widehat{w}) = \inf \{ n\geq 1 \; ; \; X_n(\widehat{w}) \in U \}$, $T_U = H_{U^c}$.}
\end{equation}
We define $H_U(\widehat{w})$ and $T_U(\widehat{w})$ in a similar fashion when $\widehat{w} \in \widehat{W}$, with ``$n\in \mathbb{Z}$'' replacing ``$n\geq 0$'' in \eqref{stop_times}, and simply write $H_x$, $\widetilde{H}_x$, $T_x$ when $U=\{x\}$. We also introduce $H_0^n(\widehat{w})$, $n \geq 1$, the successive visit times to $0$, for $\widehat{w} \in \widehat{W}_+$ or $\widehat{W}$, i.e. 
\begin{equation} \label{H_0^n}
H_0^1 = H_0, \qquad H_0^{n+1}= \left\{
\begin{array}{ll}
H_0^{n} + \widetilde{H}_0 \circ \theta_{H_0^{n}}, & \text{if } H_0^{n} < \infty \\
\infty, & \text{if } H_0^{n} = \infty
\end{array} \right.
, \quad \text{for $n \geq 1.$}
\end{equation}
\begin{comment}
\begin{lemma} Let $\rho = P_0[\widetilde{H}_0 < \infty]$. For all $n \geq 1$, 
\begin{equation} \label{visit_proba} 
P_0[H_0^{n} < \infty] = \rho^{n-1}.
\end{equation}
\end{lemma}
\begin{proof}
We proceed by induction over $n$. For $n=1$, \eqref{visit_proba} is immediate by definition of $H_0^1$ in \eqref{H_0^n}. Assume now \eqref{visit_proba} holds for some $n \in \mathbb{N}$. Then
\begin{equation*}
\begin{array}{lcl}
P_0[H_0^{n+1} < \infty] \hspace{-1ex} & \stackrel{\text{\eqref{H_0^n}}}{=} & \hspace{-1ex} P_0[H_0^{n} < \infty \;, \widetilde{H}_0 \circ \theta_{H_0^{n}} < \infty ] = P_0 \big[H_0^{n} < \infty \;, \; P_{X_{H_0^{n}}} [ \widetilde{H}_0 < \infty ] \big] \\
 & = & \hspace{-1ex} \rho \cdot P_0[H_0^{n} < \infty] = \rho^n,
\end{array}
\end{equation*}
where we have used the strong Markov property at time $H_0^{n}$ in the second step and the induction hypothesis in the last. 
\end{proof} 
\end{comment}
A straightforward application of the strong Markov property at time $H_0^{n}$ (together with an inductive argument) yields
\begin{equation} \label{visit_proba} 
P_0[H_0^{n+1} < \infty] = \rho^{n}, \quad \text{for all $n\geq 0$, where $\rho = P_0[\widetilde{H}_0 < \infty]$}.
\end{equation}
Next, we recall some basic notions from potential theory. Given some subset $K \subset \subset \mathbb{Z}^d$, we write 
\begin{equation} \label{equ_meas}
e_{K} (x) = P_x [\widetilde{H}_K = \infty], \qquad x \in K,
\end{equation}
for the equilibrium measure of $K$, and
\begin{equation}\label{cap}
\text{cap}(K) = \sum_{x \in K} e_{K} (x)
\end{equation}
for its capacity. We further denote by $\tilde{e}_K(\cdot) = e_K(\cdot) /\text{cap}(K)$ the normalized equilibrium measure. If $K=B(0,L)$ is a box of side length $L\geq 1$, one has (see for example Section 1 in \cite{SS} for a derivation)
\begin{equation} \label{cap_box}
\text{cap}(B(0,L)) \geq cL^{d-2},  \text{ for } L\geq 1.
\end{equation}

We now turn to the description of continuous-time random interlacements on $\mathbb{Z}^d$, $d \geq 3$. We define $\widehat{W}^*$ as the space $\widehat{W}$ modulo time-shift, i.e. $\widehat{W}^*= \widehat{W} / \sim $, where for any $\widehat{w},\widehat{w}' \in \widehat{W}$, $\widehat{w} \sim \widehat{w}'$ if and only if $\widehat{w}(\cdot)=\widehat{w}'(\cdot + k)$ for some $k \in \mathbb{Z}$. We denote by $\pi^*:\widehat{W} \longrightarrow \widehat{W}^*$ the corresponding canonical projection, and endow $\widehat{W}^*$ with the largest $\sigma$-algebra $\mathcal{\widehat{W}}^*$ that renders $\pi^*: (\widehat{W},\mathcal{\widehat{W}}) \longrightarrow (\widehat{W}^*, \widehat{\mathcal{W}}^*)$ measurable. For any $K \subset \subset \mathbb{Z}^d$, we let $\widehat{W}_K$ stand for the subset of $\widehat{W}$ consisting of all those doubly infinite sequences for which the $\mathbb{Z}^d$-valued trajectory enters $K$, write $\widehat{W}_K^0= \widehat{W}_K \cap \{ H_K =0 \}$, and define $\widehat{W}_K^* = \pi^*(\widehat{W}_K) \ (= \pi^*(\widehat{W}_K^0))$. If $K=\{x\}$, we simply write $\widehat{W}_x$ and $\widehat{W}_x^*$.

The continuous-time interlacement point process on $\mathbb{Z}^d$ is a Poisson point process on the space $\widehat{W}^* \times \mathbb{R}_+$. Its intensity measure is of the form $\hat{\nu}(d\widehat{w}^*)du$, where $\hat{\nu}$ is a $\sigma$-finite measure on $\widehat{W}^*$ defined as follows. For all $K\subset\subset \mathbb{Z}^d$, the restriction of $\hat{\nu}$ to $W_K^*$ is the image under $\pi^*$ of the finite measure $\widehat{Q}_K$ on $\widehat{W}$  specified by
\begin{equation} \label{intensity}
\begin{split}
\begin{array}{ll}
\text{i)} &  \widehat{Q}_K(X_0 = x) = e_K(x), \text{ for $x \in \mathbb{Z}^d$,} \\
\text{ii)} & \text{when $e_K(x) >0$, conditionally on $\{ X_0 = x \}$, $(X_n)_{n \geq 0}$, $(X_{-n})_{n\geq 0}$ and} \\
& \text{$(\sigma_n)_{n \in \mathbb{Z}}$ are independent, and respectively distributed as simple random} \\
& \text{walk starting at $x$, simple random walk starting at $x$ conditioned on not} \\
& \text{returning to $K$, and as a doubly infinite sequence of independent}\\
& \text{exponential variables with parameter one.} \\
\end{array} 
\end{split}
\end{equation}
One verifies as in the case of discrete-time random interlacements (cf. \cite{S1}, Theorem 1.1) that \eqref{intensity} defines a unique $\sigma$-finite measure $\hat{\nu}$ on $\widehat{W}^*$. In certain instances, it will be advantageous to view $\widehat{w} \in \widehat{W}$ as $(X(\widehat{w}), \sigma(\widehat{w})) \in W \times T$, where $W$ is the space of doubly infinite nearest-neighbor trajectories in $\mathbb{Z}^d$ spending finite time in finite subsets of $\mathbb{Z}^d$ (this is consistent with the notation from \cite{S1}) and $T$ is the space of doubly infinite $(0,\infty)$-valued sequences with infinite forward and backward sums. Accordingly, $\widehat{Q}_K$ becomes the product measure $Q_K \otimes P_T$, with $Q_K$ as defined in $(1.24)$ of \cite{S1} and $P_T$ a probability under which the elements of $T$ are distributed as doubly infinite sequences of independent exponential variables with parameter one. 

The continuous-time interlacement point process is then constructed on a probability space $(\Omega, \mathcal{A}, \mathbb{P})$ similar to $(1.16)$ of \cite{S1}, with $\Omega$ a space of point measures on $\widehat{W}^* \times \mathbb{R}_+$ and $\omega = \sum_{i\geq 0}\delta_{(\widehat{w}_i^*,u_i)}$ denoting a generic element of $\Omega$. The \textit{interlacement at level} $u\geq0$, denoted by $\mathcal{I}^u$, is defined as the (random) subset of $\mathbb{Z}^d$ consisting of all sites visited by at least one of the trajectories in the cloud $\omega$ with label at most $u$, i.e.
\begin{equation}\label{I^u}
\mathcal{I}^u (\omega)= \bigcup_{u_i \leq u} \text{range}(X(\widehat{w}_i)), \quad \text{if } \omega = \sum_{i\geq 0} \delta_{(\widehat{w}_i^*,u_i)},
\end{equation}
where $\widehat{w}_i$ is an arbitrary element in the equivalence class $\widehat{w}_i^*$, and $\text{range}(X(\widehat{w}_i))=\{X_n(\widehat{w}_i)\; ; n\in \mathbb{Z}\}$. Its complement $\mathcal{V}^u(\omega) = \mathbb{Z}^d\setminus \mathcal{I}^u(\omega)$ is called the \textit{vacant set at level} $u$. Note that these definitions do not depend on the exponential holding times $\sigma_n(\widehat{w}_i)$, $n \in \mathbb{Z}$, $i \geq 0$, hence the set $\mathcal{I}^u$ in  \eqref{I^u} corresponds to the (discrete-time) random interlacement at level $u \geq 0$ introduced in $(0.7)$ of \cite{S1}. For $K\subset\subset \mathbb{Z}^d$, we introduce the random point process (on $\widehat{W}_+\times\mathbb{R}_+$)
\begin{equation}\label{mu_K_def}
\hat{\mu}_K(\omega)=\sum_{i\geq 0} \delta_{(s_K(\widehat{w}_i^*)_+,u_i)} 1\{\widehat{w}_i^* \in \widehat{W}_K^*\}, \quad \text{if $\omega = \sum_{i\geq 0}\delta_{(\widehat{w}_i^*,u_i)}$},
\end{equation}
where, given some $\widehat{w}^*\in\widehat{W}_K^*$, $s_K(\widehat{w}^*)$ stands for the unique element $\widehat{w}^0$ in $\widehat{W}_K^0$ satisfying $\pi^*(\widehat{w}^0)= \widehat{w}^*$, and for arbitrary  $\widehat{w} \in \widehat{W}$, $\widehat{w}_+$ denotes the element of $\widehat{W}_+$ obtained by restricting $\widehat{w}$ to $\mathbb{N}$ (the ``forward'' trajectory). One can then show (similarly to the proof of $(1.45)$ in Proposition 1.3 of \cite{S1}) that
\begin{equation}\label{mu_K_PPP} 
\begin{split}
&\text{under $\mathbb{P}$, $\hat{\mu}_K(\omega)$ has the law of the Poisson} \\
&\text{point process on $\widehat{W}_+ \times \mathbb{R}_+$ with intensity $P_{e_K}(d\widehat{w})du$.} 
\end{split}
\end{equation} 
Given $K\subset\subset \mathbb{Z}^d$ and $u\geq 0$, we will also consider the random point process
\begin{equation} \label{mu_K,u}
\hat{\mu}_{K,u}(\omega)(d\widehat{w}) = \hat{\mu}_K(\omega)(d\widehat{w}\times [0,u])
\end{equation} 
on the space $\widehat{W}_+$. In words, $\hat{\mu}_{K,u}$ is obtained from $\hat{\mu}_K$ (as defined in \eqref{mu_K_def}) by keeping only the trajectories with label at most $u$ and forgetting their label. Similarly to \eqref{mu_K_PPP}, one deduces that the law of $\hat{\mu}_{K,u}(\omega)$, under $\mathbb{P}$, is that of the Poisson point process on $\widehat{W}_+$ with intensity $uP_{e_K}(d\widehat{w})$. In particular, this measure is finite (its total mass is $u \cdot \text{cap}(K)$), hence 
\begin{equation} \label{mu_K,u_constr}
\hat{\mu}_{K,u} \stackrel{\text{law}}{=} \sum_{i=1}^{N_{K,u}} \delta_{Z_i},
\end{equation}
where $N_{K,u} \sim \text{Poi}(u\text{cap}(K))$ (the number of trajectories with label at most $u$ entering $K$) and the $Z_i$ are i.i.d. $\widehat{W}_+$-valued random elements with law $P_{\tilde{e}_K}$ (see below \eqref{equ_meas} for the definition of $\tilde{e}_K(\cdot)$), independent of $N_{K,u}$. Occasionally, we will also consider
\begin{equation} \label{omega_K,u}
 \omega_{K,u} = \sum_{i \geq 0} \delta_{(\widehat{w}_i^*,u_i)} 1\{ \widehat{w}_i^* \in \widehat{W}_K^*, \; u_i \leq u \}, \quad \text{if } \omega = \sum_{i \geq 0} \delta_{(\widehat{w}_i^*,u_i)},
\end{equation}
which is a Poisson point process with finite intensity measure $1_{\widehat{W}_K^* \times [0,u]}\cdot \hat{\nu}(d\widehat{w}^*)du$.

In this article, we are primarily interested in the field $(L_{x,u})_{x\in \mathbb{Z}^d}$ of occupation times (at level $u \geq 0$), defined as
\begin{equation}\label{L}
\begin{split}
&\text{$L_{x,u}(\omega)= \sum_{i\geq0} \sum_{n\in \mathbb{Z}} \sigma_n(\widehat{w}_i)1\{ X_n(\widehat{w}_i)=x, u_i \leq u\}$, for $x\in \mathbb{Z}^d$, $u\geq 0$,}  \\
&\text{where $\omega = \sum_{i\geq 0}\delta_{(\widehat{w}_i^*,u_i)}\in \Omega$, and $\pi^*(\widehat{w}_i)= \widehat{w}_i^*,$ for all $i\geq0$}
\end{split}
\end{equation}
(i.e., $\widehat{w}_i$ is an arbitrary element in the equivalence class $\widehat{w}_i^*$). From \eqref{I^u,alpha}, \eqref{I^u} and \eqref{L}, we immediately infer that 
\begin{equation} \label{I^u,0}
\mathcal{I}^u = \mathcal{I}^{u,0} = \{x\in \mathbb{Z}^d\; ; L_{x,u} >0\}, \ \text{ and } \mathcal{V}^u = \mathcal{V}^{u,0}, \text{ for all $u \geq 0$}. 
\end{equation}
Thus, in particular, the random sets defined in \eqref{I^u,alpha} satisfy $\mathcal{I}^{u,\alpha}\subseteq \mathcal{I}^{u}$ and $\mathcal{V}^{u,\alpha} \supseteq \mathcal{V}^{u}$, for all $\alpha \geq 0$ and $u \geq 0$. We endow the space $ \{0,1\}^{\mathbb{Z}^d}$ with its canonical $\sigma$-algebra $\mathcal{Y}$, denote by $Y_x$, $x\in \mathbb{Z}^d$, the corresponding canonical coordinates, define the (measurable) map
\begin{equation}\label{psi_u,alpha}
\psi_{u,\alpha}: \Omega \longrightarrow \{0,1\}^{\mathbb{Z}^d}, \quad \omega \longmapsto \big( 1\{L_{x,u}(\omega)> \alpha\}\big)_{x\in\mathbb{Z}^d},
\end{equation}
and consider the image measure (on $\{0,1 \}^{\mathbb{Z}^d}$) of $\mathbb{P}$ under $\psi_{u,\alpha}$,
\begin{equation}\label{Q_u,alpha}
Q_{u,\alpha} = \psi_{u,\alpha} \circ \mathbb{P}, \quad \text{for $u,\alpha \geq 0$}.
\end{equation}

\begin{lemma} \label{trans_inv} Let $t_x$, $x\in \mathbb{Z}^d$, denote the canonical shift operators on $ \{0,1\}^{\mathbb{Z}^d}$. For all $u,\alpha \geq 0$, $t_x$, $x\in \mathbb{Z}^d$, are measure-preserving transformations on $(\{0,1\}^{\mathbb{Z}^d},\mathcal{Y}, Q_{u,\alpha})$ which are ergodic.
\end{lemma}

\begin{proof}
For arbitrary $\widehat{w} \in \widehat{W}$ and $x \in \mathbb{Z}^d$, we define $\widehat{w} + x  \in \widehat{W}$ by $(\widehat{w} + x)(n)= (X_n(\widehat{w})+x, \sigma_n (\widehat{w}))$, for all $n \in \mathbb{Z}$, and write $\widehat{w}^* +x$ for $\pi^*(\widehat{w} + x)$. Given $\omega = \sum_{i\geq 0}\delta_{(\widehat{w}_i^*,u_i)}\in \Omega$ and $x\in \mathbb{Z}^d$, we let $\tau_x\omega = \sum_{i \geq 0} \delta_{(\widehat{w}_i^* - x, u_i)}$. As in the proof of $(1.28)$ and $(1.48)$ in \cite{S1}, one verifies that $\mathbb{P}$ is invariant under $\tau_x$, for any $x\in\mathbb{Z}^d$. Using \eqref{L}, one obtains that $L_{x+y,u}(\omega)= L_{y,u}(\tau_x \omega)$, for all $x,y\in\mathbb{Z}^d$, $u\geq 0$ and $\omega \in \Omega$. This yields
\begin{equation} \label{ttau}
t_x \circ \psi_{u,\alpha} = \psi_{u,\alpha} \circ \tau_x, \quad \text{for $x\in\mathbb{Z}^d$, $u \geq 0$, $\alpha \geq 0$,}
\end{equation}
hence
\begin{equation*}
t_x \circ Q_{u,\alpha}   =  (t_x \circ \psi_{u,\alpha}) \circ \mathbb{P} = ( \psi_{u,\alpha} \circ \tau_x) \circ \mathbb{P} = \psi_{u,\alpha} \circ ( \tau_x \circ \mathbb{P}) =  Q_{u,\alpha},
\end{equation*} 
where the last step follows by translation invariance of $\mathbb{P}$. The asserted ergodicity follows from the (stronger) mixing property
\begin{equation} \label{mixing}
\lim_{|x| \to \infty} Q_{u,\alpha}[A \cap t_x^{-1}(B)] = Q_{u,\alpha}[A ] \cdot Q_{u,\alpha}[B], \quad \text{for all $A,B \in \mathcal{Y}$ (and all $u,\alpha \geq 0$)}. 
\end{equation}
\begin{comment}
Indeed, taking $B=A$ in \eqref{mixing} yields the desired $ Q_{u,\alpha}[A] =  Q_{u,\alpha}[A]^2$, whence $Q_{u,\alpha}[A] \in \{ 0,1 \}$, for $A$ any translation invariant event. 
\end{comment}
By approximation, it suffices to verify \eqref{mixing} for $A,B$ depending on the coordinates in some finite set $K\subset \subset \mathbb{Z}^d$ only. Moreover, by \eqref{mu_K,u} and \eqref{L}, the local time $L_{x,u}$, for any $x \in K$, only depends on $\omega$ ``through $\hat{\mu}_{K,u}$,'' i.e. we can write, for all $u \geq 0$ and $x \in K$,  $L_{x,u} = \sum_{i=0}^N \sum_{n \geq 0} \sigma_n(\widehat{w}_i)1\{ X_n(\widehat{w}_i)=x\}$, if  $\hat{\mu}_{K,u} = \sum_{i=0}^N \delta_{\widehat{w}_i}$ for some $N \geq 0$ and $\widehat{w}_i \in \widehat{W}_+$, $0 \leq i \leq  N$. From these observations, and in view of \eqref{ttau}, we conclude that \eqref{mixing} follows from
\begin{equation} \label{mixing2}
\lim_{|x| \to \infty} \mathbb{E}[F(\hat{\mu}_{K,u})\cdot (F(\hat{\mu}_{K,u}) \circ \tau_x)] = \mathbb{E}[F(\hat{\mu}_{K,u})]^2,
\end{equation}
for any $K \subset \subset \mathbb{Z}^d$ and $[0,1]$-valued measurable function $F$ on the set of finite point measures on $\widehat{W}_+$ (endowed with its canonical $\sigma$-field). The proof of \eqref{mixing2} is the same as that of $(2.7)$ in \cite{S1} (in particular, note that the presence of exponential holding times is inconsequential for this argument, which involves solely the \textit{spatial} part of the trajectories). This completes the proof of Lemma \ref{trans_inv}.
\end{proof}

\begin{remark} \label{R:0-1law} ($0-1$ laws)

\medskip
\noindent We consider the event $A= \{ \text{there exists an infinite cluster}\}$ $\in \mathcal{Y}$, which is translation invariant. By ergodicity, letting $\Psi^{\mathcal{I}}(u,\alpha)= Q_{u,\alpha}[A]= \mathbb{P}[\mathcal{I}^{u,\alpha} \text{ contains an infinite cluster}]$, one obtains the dichotomy
\begin{equation}\label{dichotomy}
\begin{split}
\Psi^{\mathcal{I}}(u,\alpha) = \left \{
\begin{array}{rl}
 0, & \text{ if  } \eta^{\mathcal{I}}(u,\alpha)=0 \\
 1, & \text{ if  } \eta^{\mathcal{I}}(u,\alpha)>0
\end{array}
\right.
\end{split}
\end{equation}
(see \eqref{eta^I} for the definition of $\eta^{\mathcal{I}}(u,\alpha)$). In particular, this implies $\Psi^{\mathcal{I}}(u,\alpha) = 1$ for all $u > 0$ and $0 \leq \alpha < \alpha_*(u)$ (recall \eqref{alpha_*} for the definition of $\alpha_*(u)$), and $\Psi^{\mathcal{I}}(u,\alpha) = 0$ for all $u \geq 0$ and $ \alpha > \alpha_*(u)$ (note that $\alpha_*(u) \in [0,\infty]$ at this point). 

The conclusions of Lemma \ref{trans_inv} and \eqref{mixing} continue to hold if one replaces $Q_{u,\alpha}$ by $\widetilde{Q}_{u,\alpha}$, the law of $(1\{ x \in \mathcal{V}^{u,\alpha} \})_{x \in \mathbb{Z}^d}$ under $\mathbb{P}$ (using the inversion map on $\{ 0,1\}^{\mathbb{Z}^d}$). On account of this, the analogue of \eqref{dichotomy} for $\Psi^{\mathcal{V}}(u,\alpha)= \widetilde{Q}_{u,\alpha}[A]$ (obtained by replacing $\mathcal{I}$ with $\mathcal{V}$ everywhere in \eqref{dichotomy}) holds as well. Recalling \eqref{u_*}, this gives $\Psi^{\mathcal{V}}(u,\alpha) = 1$ for all $\alpha \geq 0$ and $0 \leq u < u_*(\alpha)$, and $\Psi^{\mathcal{V}}(u,\alpha) = 0$ for all $\alpha \geq 0$ and $ u > u_*(\alpha)$ (and $u_*(\alpha)$ could be infinite at this point). \hfill $\square$
\end{remark}

\section{Finite energy property and uniqueness} \label{FE}

In this section, we establish in Theorem \ref{T:FE} that the measure $Q_{u,\alpha}$ has the so-called finite energy property, see \eqref{RESULT:FINITE_ENERGY}, for all $u >0$ and $\alpha>0$. One important consequence is the (almost sure) uniqueness of the infinite cluster of  $\mathcal{I}^{u, \alpha}$ in the supercritical regime, i.e., for all $u,\alpha > 0$ such that $\eta^{\mathcal{I}}(u,\alpha) >0$, see Corollary \ref{C:UNIQUENESS} below. Analogous conclusions hold for $\mathcal{V}^{u, \alpha}$, see Remark \ref{FE_V^u,alpha}. As mentioned in the introduction, this differs noticeably from the case $\alpha = 0$, which corresponds to discrete-time random interlacements, as $Q^{u,0}$, $u \geq 0$, does \textit{not} have the finite energy property. Indeed, when $\alpha = 0$, if e.g. all the sites neighboring $0$ are vacant (i.e. not visited \textit{at all} by a random walk trajectory), then $0$ is necessarily vacant, too. When $\alpha > 0$, the vacancy of a site only requires its local time to be sufficiently small (less than $\alpha$), and, as hinted in the introduction (below \eqref{RESULT:FINITE_ENERGY}), one can still hope to force $0$ to be occupied by sending an  ``invisible'' trajectory to the origin to ensure that $L_{0,u} > \alpha$.

A key ingredient to make this strategy work is that a certain class of \textit{local} surgical operations on paths can be performed in an absolutely continuous way (with respect to $\mathbb{P}$), as we now explain. We use $\ll$ to denote absolute continuity, and begin with the following general

\begin{lemma} \label{L_AC1}
Let $\nu$ and $\nu'$ be two finite measures on a common measurable space $(W, \mathcal{W})$, and assume that $\nu' \ll \nu$. Then $\mathbb{P}_{\nu'} \ll \mathbb{P}_{\nu}$, where $\mathbb{P}_{\nu}$, $\mathbb{P}_{\nu'}$, denote the laws of the canonical Poisson point processes with intensity measures $\nu$, $\nu'$, respectively.
\end{lemma}

\begin{proof}
One verifies that the Radon-Nikodym derivative is given by (with $\omega = \sum_{i=1}^n \delta_{w_i}$ a generic finite point measure on $W$)
\begin{equation*}
\frac{d \mathbb{P}_{\nu'}}{d \mathbb{P}_{\nu}}(\omega) = e^{-( \nu'(W) - \nu(W))} \prod_{i=1}^n \frac{d\nu'}{d\nu}(w_i).
\end{equation*}
\end{proof}

The following result concerning the intensity measure of random interlacements ensures that Lemma \ref{L_AC1} applies in the context of  \textit{local} ``path surgery arguments'' of a certain kind, which we now describe. We recall (see below \eqref{intensity}) that $\widehat{W}= W \times T$, where $W$ is the space of doubly infinite, nearest-neighbor trajectories in $\mathbb{Z}^d$ spending finite time in any finite subset of $\mathbb{Z}^d$. We denote by $Z_n$, $n \in \mathbb{Z}$, the canonical coordinates on $W$. Moreover, for any $K \subset \subset \mathbb{Z}^d$, we can write $1_{\widehat{W}_K^*} \hat{\nu} = \pi^* \circ (Q_K \otimes P_T)$, where $Q_K$ is a measure on $W$, supported on $W_K^0 \subset W$, which consists of those trajectories hitting $K$ at time $0$. Finally, for arbitrary $K \subset \subset \mathbb{Z}^d$, we denote by $\mathcal{T}_K$ the set of finite-length (discrete-time), nearest-neighbor trajectories in $\mathbb{Z}^d$ starting and ending in the support of $e_K(\cdot)$ (see \eqref{equ_meas}):
\begin{equation} \label{EQ:T_K}
\begin{split}
\mathcal{T}_K =\{ 
&\tau = (\tau(n))_{0\leq n \leq N_\tau} ; \ \text{ $N_\tau \geq 0$, $\tau(n) \in \mathbb{Z}^d$, for $0\leq n \leq N_\tau$,} \\ 
&\text{$|\tau(n+1)-\tau(n)|=1$, for $0\leq n < N_\tau$, and $\tau(0), \tau(N_\tau) \in \text{supp}(e_K)$} \}.
\end{split}
\end{equation}
Let $F: \mathcal{T}_K \rightarrow \mathcal{T}_K$ be a map which preserves initial and final points, i.e. 
\begin{equation} \label{AC_F}
\text{$F(\tau)(0) = \tau(0)$ and $F(\tau)(N_{F(\tau)}) = \tau (N_\tau)$, for all $\tau = (\tau(n))_{0\leq n \leq N_\tau} \in \mathcal{T}_K$}.
\end{equation}
$F$ induces a (measurable) function $\varphi_F: W_K^0 \rightarrow W_K^0$ as follows. For any $w \in W_K^0$, $\varphi_F(w)$ is obtained by replacing the excursion of $w$ from time zero (when it hits $K$) until the time of last visit to $K$ by its image under $F$, i.e.
\begin{equation} \label{AC_phi_F}
Z_n (\varphi_F(w)) = 
\begin{cases}
Z_n(w), & n < 0,\\
\tilde{\tau}(n), & 0\leq n \leq N_{\tilde{\tau}} \\
Z_{n-N_{\tilde{\tau}} + L_K(w)}(w), & n > N_{\tilde{\tau}}
\end{cases}
, \qquad \text{if $\tilde{\tau}= F\big( (Z_n(w))_{0\leq n \leq L_K(w)} \big)$},
\end{equation}
where $L_K(w)= \sup\{ n \geq 0 ; \; Z_n(w) \in K \}$, for $w\in W_K^0$. One then obtains the following 

\begin{lemma}\label{L_AC2} $(K \subset\subset \mathbb{Z}^d)$

\medskip
\noindent For any map $F: \mathcal{T}_K \rightarrow \mathcal{T}_K$ as above, the measure $\varphi_F \circ Q_K$ is absolutely continuous with respect to $Q_K$, and
\begin{equation} \label{L_AC2.1}
 d (\varphi_F \circ Q_K )= f \cdot d Q_K, 
\end{equation} 
with
\begin{equation}  \label{L_AC2.2}
 \quad f(w)= \sum_{ \substack{\tau   \in \mathcal{T}_K: \\ F(\tau) = (w(n))_{0\leq n \leq L_K(w)} }} (2d)^{L_K(w)-N_\tau}, \text{ for $w \in W_K^0$}.
 \end{equation}
\end{lemma}
\begin{proof}
We recall (see \cite{S1}, Theorem 1.1) that the law of $(Z_n)_{0\leq n \leq L_K}$ under $Q_K$ is supported on $\mathcal{T}_K$ and 
\begin{equation} \label{L_AC2_pf1}
\begin{split}
&Q_K[(Z_{-n})_{n\geq 0} \in A, \; (Z_n)_{0 \leq n \leq L_K} =\tau, \; (Z_{n+L_K})_{n\geq 0} \in B] = \\
&P_{\tau(0)}^K[A] \cdot e_K \big(\tau(0)\big) \cdot P_{\tau(0)}[Z_n = \tau(n), \; 0 \leq n \leq N_\tau] \cdot e_K \big(\tau(N_\tau)\big)\cdot P^K_{\tau(N_\tau)}[B],
\end{split}
\end{equation}
for all $\tau \in \mathcal{T}_K$ and $A,B \in \mathcal{W}_+$, where $\mathcal{W}_+$ denotes the canonical $\sigma$-algebra on $W_+$, the space of nearest-neighbor trajectories in $\mathbb{Z}^d$ spending finite time in any finite subset of $\mathbb{Z}^d$, $P_x$, $x\in \mathbb{Z}^d$, is the restriction to $(W_+, \mathcal{W}_+)$ of the law of (discrete-time) simple random walk on $\mathbb{Z}^d$, and $P_x^K[\cdot]=P_x[\cdot| \widetilde{H}_K = \infty]$ (we will use $Z_n$, $n \geq 0$, to denote canonical coordinates on $W_+$). For arbitrary sets $A_n \subset \mathbb{Z}^d$, $n\in \mathbb{Z}$, we have
\begin{align*}
&(\varphi_F \circ Q_K) [Z_n \in A_n, \; n \in \mathbb{Z}] =  Q_K [(Z_n \circ \varphi_F) \in A_n, \; n  \in \mathbb{Z}] = \\
&\sum_{\tau \in \mathcal{T}_K} Q_K [(Z_n \circ \varphi_F) \in A_n, \; n  \in \mathbb{Z}, \;  (Z_n)_{0\leq n \leq L_K}= \tau] \stackrel{\eqref{AC_phi_F}}{=} \\
&\sum_{\substack{\tau \in \mathcal{T}_K : \;  F(\tau)(n) \in A_n, \\ 0\leq n \leq N_{F(\tau)}}} Q_K [Z_n  \in A_n, \; n  < 0, \; Z_{n-N_{F(\tau)}+N_\tau}  \in A_n, \; n> N_{F(\tau)}, \; (Z_n)_{0\leq n \leq L_K}= \tau] \stackrel{\eqref{L_AC2_pf1}}{=} 
\end{align*}
\begin{flalign*}
&\sum_{\substack{\tau \in \mathcal{T}_K : \; F(\tau)(n) \in A_n, \\
0\leq n \leq N_{F(\tau)}}} P_{\tau(0)}^K[Z_n  \in A_n, \; n  > 0]\cdot e_K \big(\tau(0) \big) \cdot (2d)^{-N_\tau} \cdot e_K \big(\tau(N_\tau)\big)& \\
&\qquad \qquad \qquad \quad \times P^K_{\tau(N_\tau)}[Z_n \in A_{n + N_{F(\tau)}}, \; n >0] \stackrel{\eqref{AC_F}}{=}&
\end{flalign*}
\begin{flalign*}
&\sum_{\substack{\tau \in \mathcal{T}_K : \; F(\tau)(n) \in A_n, \\
0\leq n \leq N_{F(\tau)}}} P_{F(\tau)(0)}^K[Z_n  \in A_n, \; n  > 0]\cdot e_K \big(F(\tau)(0) \big) \cdot (2d)^{-N_{F(\tau)}} \cdot e_K \big(F(\tau)(N_{F(\tau)})\big)& \\
&\qquad \qquad \qquad \quad \times P^K_{F(\tau)(N_{F(\tau)})}[Z_n \in A_{n + N_{F(\tau)}}, \; n >0] \cdot (2d)^{N_{F(\tau)} - N_\tau} \stackrel{\eqref{L_AC2_pf1}}{=}&
\end{flalign*}
\begin{flalign*}
&\sum_{\substack{\tau \in \mathcal{T}_K : \;  F(\tau)(n) \in A_n, \\
0\leq n \leq N_{F(\tau)}}} Q_K[Z_n  \in A_n, \; n  < 0, (Z_n)_{0 \leq n \leq L_K} = F(\tau), \; (Z_{n+L_K})_{n \geq 0} \in A_{n+N_{F(\tau)}}, \; n>0 ]& \\
&\qquad \qquad \qquad \quad \times (2d)^{N_{F(\tau)} - N_\tau} \stackrel{\tilde{\tau}= F(\tau)}{=}&
\end{flalign*}
\begin{flalign*}
&\sum_{\tilde{\tau} \in \mathcal{T}_K} Q_K[Z_n  \in A_n, \; n  \in \mathbb{Z}, \; (Z_n)_{0 \leq n \leq L_K} = \tilde{\tau}] \cdot \sum_{\substack{\tau \in \mathcal{T}_K: \\ F(\tau) = \tilde{\tau} }} (2d)^{N_{\tilde{\tau}} -N_\tau} = & \\
&E^{Q_K}[f \; ; \; Z_n \in A_n, \; n \in \mathbb{Z}],&
\end{flalign*}
where, in the penultimate line, we adopt the convention that a sum over an empty indexing set is equal to zero (no contribution arises if $\tilde{\tau} \notin \text{Im}(F)$), and $f$ is defined in \eqref{L_AC2.2}. To see that $f$ is finite, we note that $L_K(w) < \infty$ for all $w \in W_K^0$ ($K$ is a finite set), and observe that, on account of \eqref{AC_F} and \eqref{L_AC2.2}, for every $w \in W_K^0$, setting $x_i = w(0)$ and $x_e= w(L_K(w))$,
\begin{align*}
f(w) &\leq \sum_{ \substack{\tau   \in \mathcal{T}_K: \\ \tau(0)=x_i, \tau(N_\tau)=x_e }} (2d)^{L_K(w)-N_\tau} \\ &= \sum_{ \substack{\tau   \in \mathcal{T}_K: \\ \tau(0)=x_i, \tau(N_\tau)=x_e }} (2d)^{L_K(w)} \cdot P_{x_i}[(Z_n)_{0\leq n \leq N_\tau} = \tau] \leq (2d)^{L_K(w)} (1 + P_{x_i}[\widetilde{H}_{x_e} < \infty]) < \infty.
\end{align*}
We conclude that $f$ is the Radon-Nikodym derivative of $\varphi_F \circ Q_K$ with respect to $Q_K$, i.e., \eqref{L_AC2.1} holds, and in particular, $\varphi_F \circ Q_K \ll Q_K$. This completes the proof of Lemma \nolinebreak \ref{L_AC2}.
\end{proof}

With Lemma \ref{L_AC2} at hand, we proceed to the main result of this section, the finite-energy property of $Q_{u,\alpha}$, for $u,\alpha >0$.

\begin{theorem} \label{T:FE} Let $Y_z$, $z \in \mathbb{Z}^d$, denote the canonical coordinates on $\{ 0,1\}^{\mathbb{Z}^d}$. For all $u > 0$, $\alpha >0$ and $x\in \mathbb{Z}^d$,
\begin{equation} \label{T:FE1}
0< Q_{u,\alpha}(Y_x =1 | \sigma( Y_z, \; z \neq x) )<1, \quad \text{ $Q_{u,\alpha}$-a.s.}
\end{equation}
\end{theorem} 

\begin{proof}
Let $u> 0$, $\alpha >0$. By translation invariance of $Q_{u,\alpha}$, see Lemma \ref{trans_inv}, it suffices to prove \eqref{T:FE1} for $x=0$. The latter amounts to showing that for all $A \in \sigma(Y_z \; ; z \in \mathbb{Z}^d \setminus \{ 0 \})$ with $Q_{u,\alpha}[A]>0$, one has
\begin{align} 
&Q_{u,\alpha}[A \cap \{ Y_0 = 1\}]>0, \label{T:FE12} \\
&Q_{u,\alpha}[A \cap \{ Y_0 = 0\}]>0.  \label{T:FE11}
\end{align}
We begin by proving \eqref{T:FE12}, which is the more difficult part and involves the ``path surgery'' argument described at the beginning of this section (cf. also the discussion following \eqref{RESULT:FINITE_ENERGY}).  We recall the definition \eqref{psi_u,alpha} of the map $\psi_{u,\alpha}$, and define $A_{u,\alpha}= \psi_{u,\alpha}^{-1}(A)$, so that $\mathbb{P}[A_{u,\alpha}] = Q_{u,\alpha}[A] >0$. Thus, we need to show that
\begin{equation} \label{FE2proof0}
\mathbb{P}[A_{u,\alpha}, L_{0,u} > \alpha] >0.
\end{equation}
For arbitrary $K\subset \subset \mathbb{Z}^d$, we write $N_{K,u}$ for the number of trajectories (modulo time-shift) in the interlacement with label at most $u$ which visit $K$, i.e. $N_{K,u}= \omega \big(\widehat{W}_K^* \times [0,u]\big)$. By \eqref{intensity}, $N_{K,u}$ has a Poisson distribution with parameter $u \cdot \text{cap}(K)$, so in particular, $\mathbb{P}[N_{K,u} > 0] = 1 - \exp(-u \text{cap}(K))$. Using \eqref{cap_box}, we can thus find $K=B(0,L)\subset \subset \mathbb{Z}^d$, with $L\geq 1$ sufficiently large, and $\varepsilon >0$ small enough such that the event $A_{u,\alpha}^\varepsilon \stackrel{\text{def.}}{=} A_{u,\alpha} \cap \{ N_{K,u} >0 \} \cap \bigcap_{x\in K} \{ L_{x,u} \notin (\alpha-\varepsilon ,\alpha] \}$ has positive probability under $\mathbb{P}$, and therefore
\begin{equation}\label{FE2proof1}
\mathbb{P}[A_{u,\alpha}^\varepsilon, N_{K,u}=N] >0,
\end{equation}
for some $N\geq 1$. We recall that $\mathcal{T}_K$ is the set of finite nearest-neighbor trajectories on $\mathbb{Z}^d$ starting and ending in $\text{supp}(e_K) \; (= \partial^i K)$, cf. \eqref{EQ:T_K}, and, given some $\tau \in \mathcal{T}_K$, define $\widehat{W}_\tau^* = \pi^* ( \widehat{W}_\tau^0 )$, where
\begin{equation} \label{FE2proof2}
\widehat{W}_{\tau}^0 =  \big\{ \widehat{w} \in \widehat{W}_K^0 \; ; \;  (X_n(\widehat{w}))_{0\leq n \leq L_K(\widehat{w})} = \tau \big\}  \in \widehat{\mathcal{W}},
\end{equation} 
with $L_K(\widehat{w})= \sup\{ n\geq 0 ; \; X_n(\widehat{w})\in K \}$. For arbitrary $\tau \in \mathcal{T}_K$, we consider
\begin{equation*}
D(N,\tau,u) = \big\{ N_{K,u}=N, \;  \omega \big(\widehat{W}_{\tau}^* \times[0,u] \big) >0 \big\},
\end{equation*}
the event that $N$ trajectories (modulo time-shift) with label at most $u$ enter $K$ and the trace left on $\mathbb{Z}^d$ by at least one of them from the time it first enters $K$ until its time of last visit to $K$ is precisely given by $\tau$. Note that $\bigcup_{\tau \in \mathcal{T}_K} D(N, \tau,u)= \{ N_{K,u} = N \}$. Thus, on account of \eqref{FE2proof1}, and since $\mathcal{T}_K$ is a countable set, we may select $\tau \in \mathcal{T}_K$ such that
\begin{equation} \label{FE2proof3}
\mathbb{P}\big[A_{u,\alpha}^\varepsilon, D(N, \tau, u) \big] >0,
\end{equation} 
and consider this $\tau  = (\tau(n))_{0\leq n \leq N_\tau}$, with $0\leq N_\tau < \infty$, to be fixed from now on. Next, we let $\bar{\tau} \in \mathcal{T}_K$ be a closed finite nearest-neighbor path starting and ending in $\tau(0)$, i.e., satisfying $\bar{\tau}(0)=\bar{\tau}(N_{\bar{\tau}})=\tau(0)$, and passing through the origin. Furthermore, we assume that $\text{range}(\bar{\tau}) \subset K$ and that $\bar{\tau}$ visits each vertex in $ \text{range}(\bar{\tau}) \setminus\{\bar{\tau}(0) \}$ exactly once (this can always be arranged). Viewing the set $\widehat{W}_\tau^0$ defined in \eqref{FE2proof2} as $W_\tau^0 \times T$ (see the discussion below \eqref{intensity}), we define the (measurable) map 
\begin{equation} \label{FE2proof_phi}
\begin{split}
\varphi: W \rightarrow W, \quad &\text{where $\varphi$ acts as identity on $W\setminus W_{\tau}^0$, and for $w \in W_{\tau}^0$, $\varphi (w)$ is}\\
&\text{the path obtained by ``inserting'' $\bar{\tau}$ when $w \in W_{\tau}^0$ hits $K$,}
\end{split}
\end{equation}
i.e. such that, given $w \in W_{\tau}^0$, $X_n(\varphi (w))= X_n(w)$, for all $n \leq 0$, $X_n(\varphi (w))= \bar{\tau} (n)$, for $1 \leq n \leq N_{\bar{\tau}}$, and $X_n(\varphi (w))= X_{n-N_{\bar{\tau}}}(w)$, for $n> N_{\bar{\tau}}$ (in particular, observe that $\varphi(W_{\tau}^0) \subset W_K^0$). We also introduce an auxiliary probability $P_{\Lambda}$ on the space $\Lambda= \mathbb{R}_+^{\mathbb{N}_*} \ni (\lambda_n)_{ n \geq 1}$, such that the canonical coordinates on $\Lambda$ are distributed as independent exponential variables with parameter one. We now define, for all $\lambda = (\lambda_n)_{ n \geq 1} \in \Lambda$ and $i \geq 1$, the maps
\begin{equation*}
\varphi_{\lambda,i} : W_{\tau}^0 \times T \rightarrow W_K^0 \times T, \quad \varphi_{\lambda,i} (w,t) = \big( \varphi(w), \;  \theta_{\lambda,i} (t) \big)
\end{equation*}
where $\theta_{\lambda,i} (t)(n) = t(n)$ for $n \leq 0$, $\theta_{\lambda,i} (t)(n) = \lambda_{n+ (i-1) \cdot N_{\bar{\tau}} }$ for $ 1 \leq n \leq N_{\bar{\tau}}$ and $\theta_{\lambda,i} (t)(n) = t(n-N_{\bar{\tau}})$ for $n > N_{\bar{\tau}}$. In words, the effect of $\varphi_{\lambda,i}$ is to add the piece of path $\bar{\tau}$ when $\widehat{w} \in \widehat{W}_{\tau}^0$ hits $K$ and to ``inject'' the holding times $\lambda_{n+(i-1) \cdot N_{\bar{\tau}} }$, $1 \leq n \leq N_{\bar{\tau}} $, underneath it. We further define $\varphi_{\lambda,i}^*: \widehat{W}_{\tau}^* \rightarrow \widehat{W}_K^*$ by $\varphi_{\lambda,i}^* \circ \pi^* = \pi^* \circ \varphi_{\lambda,i}$, for all $\lambda \in \Lambda$ and $i \geq 0$, and extend $\varphi_{\lambda,i}^*$ to $\widehat{W}_K^*$ by letting it act as identity on $\widehat{W}_K^* \setminus \widehat{W}_{\tau}^* $. Finally, we introduce the map $\Phi_{\bar{\tau}}: \Lambda \times \Omega \rightarrow \Omega$,
\begin{equation} \label{L:FE2proof1}
\Phi_{\bar{\tau}} (\lambda, \omega) = \widetilde{\Phi}_{\bar{\tau}} (\lambda, \omega_{K,u}) + (\omega-\omega_{K,u}),
\end{equation}
with $\omega_{K,u}$ as defined in \eqref{omega_K,u}, where
\begin{equation} \label{FE2proof6}
\widetilde{\Phi}_{\bar{\tau}} ( \lambda,  \omega_{K,u} ) = \sum_{i=1}^{n} \delta_{(\varphi_{\lambda,i}^*(\widehat{w}_i^*),u_i)}, \quad \text{if} \quad \omega_{K,u} = \sum_{i=1}^{n} \delta_{(\widehat{w}_i^*,u_i)},
\end{equation}
where we assume for definiteness that the points $(\widehat{w}_i^*,u_i)$ in the support of $\omega_{K,u}$ are ordered according to increasing $u_i$ (this yields a well-defined map $\Phi_{\bar{\tau}}$ on a subset of $\Lambda \times \Omega$ having full measure under $P_\Lambda \otimes \mathbb{P}$, since the $u_i$'s are almost surely different, see \eqref{intensity}). From the preceding construction, we obtain the following  

\begin{lemma} \label{L:FE2_AbsCont} 
The measure $\mathbb{P}_{\bar{\tau}} \stackrel{\text{def.}}{=} \Phi_{\bar{\tau}} \circ (P_\Lambda \otimes \mathbb{P})$ is the law of the Poisson point process on $\widehat{W}^* \times \mathbb{R}_+$ with intensity
\begin{equation} \label{L:FE2_AbsCont1}
 d\hat{\nu}_{K}^{\varphi} \cdot 1_{[0,u]}dl + 1_{(\widehat{W}_{K}^* \times [0,u])^c} d\hat{\nu} dl, \quad \text{where $ \;  
\hat{\nu}_K^\varphi= \pi^* \circ \big((\varphi \circ  Q_K) \otimes P_T \big)$}
\end{equation}
$($see below \eqref{intensity} for notation$)$, with $l$ denoting Lebesgue measure on $\mathbb{R}_+$ and $\varphi$ as defined in \eqref{FE2proof_phi}. Moreover,
\begin{equation} \label{L:FE2_AbsCont2}
 \mathbb{P}_{\bar{\tau}} \ll \mathbb{P}.
\end{equation}
\end{lemma}
\begin{proof2} 
First, we note that $\omega_{K,u}$ and $\omega- \omega_{K,u}$ are independent. Hence, by \eqref{L:FE2proof1} and the superposition principle for Poisson point processes, \eqref{L:FE2_AbsCont1} follows at once if we show that under $P_\Lambda \otimes \mathbb{P}$,
\begin{equation} \label{L:FE2proof2}
\text{$\widetilde{\Phi}_{\bar{\tau}} (\lambda, \omega_{K,u})$ is a Poisson point process  with intensity $d\hat{\nu}_K^\varphi \cdot 1_{[0,u]}dl$}.
\end{equation}
To see the latter, one may argue constructively as follows. Using an explicit representation of the process
\begin{comment}
(its total mass is $u \cdot \hat{\nu}(\widehat{W}_\tau^*)= u\cdot e_K(\tau(0)) \cdot(2d)^{-N_\tau} \cdot e_K(\tau(N_\tau))$, cf. (1.26) in \cite{S1}) 
\end{comment}
$\omega_{K,u}$ (similar to \eqref{mu_K,u_constr} for $\hat{\mu}_{K,u}$ as defined in \eqref{mu_K,u}), \eqref{L:FE2proof2} is immediate upon observing that if $Z_i^*$, $i\geq 1$, are independent, $\widehat{W}_K^*$-valued random elements with distribution $ 1_{\widehat{W}_K^*} \hat{\nu} / \text{cap}(K)$ under some probability $P$, then $\varphi_{\lambda,i}^*(Z_i^*)$, $i\geq 1$, are independent and all distributed according to $\hat{\nu}_K^\varphi / \text{cap}(K)$ under $P\otimes P_\Lambda$.

It remains to show \eqref{L:FE2_AbsCont2}, which will involve Lemma \ref{L_AC2}. We write $\mathbb{P}= G \circ (\mathbb{P}_{K,u} \otimes \overline{\mathbb{P}})$, where $\mathbb{P}_{K,u}$ and $\overline{\mathbb{P}}$ denote the laws of $\omega_{K,u}$ and $\omega-\omega_{K,u}$, respectively, and $G(\omega, \omega')=\omega + \omega'$, so that
\begin{equation}\label{L:FE2proof3}
\mathbb{P}_{\bar{\tau}} = \Phi_{\bar{\tau}} \circ \big[P_\Lambda \otimes \big(G \circ (\mathbb{P}_{K,u} \otimes \overline{\mathbb{P}})\big)\big] \stackrel{\eqref{L:FE2proof1}}{=} G \circ \big[ \big( \widetilde{\Phi}_{\bar{\tau}} \circ(P_\Lambda \otimes\mathbb{P}_{K,u})\big) \otimes \overline{\mathbb{P}}].
\end{equation}
By Lemma \ref{L_AC2}, we obtain that $\varphi \circ Q_K \ll Q_K$ (indeed, the map $\varphi$ defined in \eqref{FE2proof_phi} can be written as $\varphi_F$, with $F: \mathcal{T}_K \rightarrow \mathcal{T}_K$ preserving initial and final points, i.e. satisfying \eqref{AC_F}). Hence, $\hat{\nu}_K^\varphi \ll 1_{\widehat{W}_K^*} \hat{\nu} \ (= \pi^* \circ (Q_K \otimes P_T))$, which, using Lemma \ref{L_AC1} and \eqref{L:FE2proof2}, yields $\big( \widetilde{\Phi}_{\bar{\tau}} \circ(P_\Lambda \otimes \mathbb{P}_{K,u})\big) \ll \mathbb{P}_{K,u}$. Together with \eqref{L:FE2proof3}, this implies \eqref{L:FE2_AbsCont2}, and thus completes the proof of Lemma \ref{L:FE2_AbsCont}. \hfill $\square$
\end{proof2}

\medskip
\noindent We now conclude the proof of \eqref{T:FE12} and recall to this end the definition of $A_{u,\alpha}^\varepsilon \cap D(N,\tau,u)$ in \eqref{FE2proof3}. Denoting by $N_{\bar{\tau}}^0 \in \{1,\dots, N_{\bar{\tau}} \}$ the (only) time at which $\bar{\tau}$ visits the origin, i.e. $\bar{\tau}(N_{\bar{\tau}}^0)=0$, we consider the cylinder set
\begin{equation*}
\mathcal{C}= \big( [0, \varepsilon/N]^{N_{\bar{\tau}}^0 -1} \times (\alpha,\infty)\times [0, \varepsilon/N]^{ N_{\bar{\tau}} - N_{\bar{\tau}}^0} \big)^N \times \mathbb{R}^{\mathbb{N}_* \setminus \{1,\dots,N\cdot N_{\bar{\tau}} \}} \subset \Lambda, 
\end{equation*}
and claim that
\begin{equation} \label{FE2proof7}
\Phi_{\bar{\tau}}^{-1} \big( A_{u,\alpha} \cap \{L_{0,u} > \alpha \}\big) \supset \big(\mathcal{C} \times (A_{u,\alpha}^\varepsilon \cap D(N,\tau,u))\big).
\end{equation}
Indeed, let $ \lambda \in \mathcal{C}$ and $\omega \in A_{u,\alpha}^\varepsilon \cap D(N,\tau,u)$. In particular, at least one trajectory with label at most $u$ in the support of $\omega$ belongs to $\widehat{W}_\tau^*$. By definition, $\Phi_{\bar{\tau}}$ ``adds'' the piece of path $\bar{\tau}$ to each such trajectory when it first hits $K$. This implies that some trajectory with label at most $u$ in the support of $\Phi_{\bar{\tau}} (\lambda,\omega)$ visits $0$, and thus $L_{0,u}(\Phi_{\bar{\tau}} (\lambda,\omega)) > \alpha$, by definition of $\mathcal{C}$. Next, assume $x \in K \setminus \{ 0 \}$ and $L_{x,u}(\omega) \leq \alpha$. Then in fact $L_{x,u}(\omega) \leq \alpha -\varepsilon$, and  $L_{x,u}(\Phi_{\bar{\tau}} (\lambda,\omega)) \leq L_{x,u}(\omega) + N \cdot \frac{\varepsilon}{N} \leq \alpha$, since at most all $N$ trajectories with label at most $u$ hitting $K$ belong to $\widehat{W}_\tau^*$, and any vertex in $K$ is visited by $\bar{\tau}$ at most once. On the other hand, if $x\in K\setminus \{ 0 \}$ and $L_{x,u}(\omega) > \alpha$, then clearly $L_{x,u}(\Phi_{\bar{\tau}} ( \lambda ,\omega))\geq L_{x,u}(\omega) > \alpha$, for, by construction, $\Phi_{\bar{\tau}}$ can only increase the local time at any point. Finally, since the range of $\bar{\tau}$ is contained in $K$, $L_{x,u}(\Phi_{\bar{\tau}} (\lambda,\omega)) = L_{x,u}(\omega)$ for any $x \in K^c$. All in all, this yields $1\{ L_{x,u} (\Phi_{\bar{\tau}} (\lambda,\omega)) > \alpha \}= 1\{ L_{x,u} (\omega) > \alpha \}$, for all $x\in \mathbb{Z}^d \setminus \{ 0 \}$, i.e. $\Phi_{\bar{\tau}} (\lambda ,\omega) \in A_{u,\alpha}$. Thus, \eqref{FE2proof7} holds. Finally, we obtain
\begin{equation*}
\begin{array}{rcl}
\mathbb{P}_{\bar{\tau}}[A_{u,\alpha} , \;  L_{0,u} > \alpha ]  & \hspace{-1ex} =  & \hspace{-1ex} (P_\Lambda \otimes \mathbb{P}) \big[\Phi_{\bar{\tau}}^{-1} \big( A_{u,\alpha} \cap \{L_{0,u} > \alpha \} \big)\big] \\
& \hspace{-1ex} \stackrel{\eqref{FE2proof7}}{\geq} & \hspace{-1ex} P_\Lambda[\mathcal{C}] \cdot \mathbb{P}\big[A_{u,\alpha}^\varepsilon, D(N, \tau,u) \big] \\
& \hspace{-1ex} \stackrel{\eqref{FE2proof3}}{>} & \hspace{-1ex} 0,
\end{array}
\end{equation*}
which, together with \eqref{L:FE2_AbsCont2}, implies \eqref{FE2proof0}, and thus completes the proof of \eqref{T:FE12}. 

\bigskip
\noindent We now turn to the proof of \eqref{T:FE11}, which is considerably simpler (roughly speaking, we can easily \textit{prevent} the local time at $0$ from being too large once we control how many trajectories hit the origin and how often each one of them returns to $0$).  Let $u,\alpha >0$, $A \in \sigma(Y_z \; ; \; z\neq 0 )$ satisfy $Q_{u,\alpha}[A]>0$, and let $A_{u,\alpha} = \psi_{u,\alpha}^{-1}(A)$, as above. We denote by $N_{0,u}$ the number of trajectories (modulo time-shift) with label at most $u$ visiting the origin. Since $N_{0,u}$ is integer-valued, there exists $N \in \mathbb{N}$ such that
\begin{equation} \label{FE1proof2}
\mathbb{P}[A_{u,\alpha},N_{0,u} =N] >0.
\end{equation}
We recall the definition \eqref{H_0^n} of $H_0^n$, $n\geq 1$, the successive times a trajectory visits $0$, and introduce the (increasing) family of measurable subsets (part of $\mathcal{W}_+$)
\begin{equation} \label{V_M}
\widehat{V}_+^n = \{ \widehat{w} \in \widehat{W}_+ \; ; X_0(\widehat{w})=0 \text{ and } H_0^{n+1}(\widehat{w}) = \infty \}, \quad \text{for $n \geq 1$}.
\end{equation}
In words, $\widehat{V}_+^n$ consists of all trajectories in $\widehat{W}_+$ starting in $0$ which visit $0$ at most $n$ times. By virtue of \eqref{mu_K,u_constr}, we obtain that
\begin{equation} \label{FE1proof3}
\mathbb{P}\big[\hat{\mu}_{\{0\},u}(\widehat{V}_+^n) = N \big|N_{0,u} =N \big] =\big( P_0[H_0^{n+1}= \infty] \big)^N \stackrel{\text{\eqref{visit_proba}}}{=} (1-\rho^n)^N \xrightarrow{n\to\infty}1,
\end{equation}
where the last step is due to transience, and $\rho=P_0[\widetilde{H}_0 < \infty]<1$. Since $\mathbb{P}[A_{u,\alpha} | N_{0,u} =N]>0$ by \eqref{FE1proof2}, it follows from \eqref{FE1proof3} that there exists $M \geq 1$ such that $\mathbb{P}\big[A_{u,\alpha}, \hat{\mu}_{\{0\},u}(\widehat{V}_+^M) = N \big|N_{0,u} =N \big] >0$, and thus also
\begin{equation}\label{FE1proof4}
\mathbb{P}\big[A_{u,\alpha} \big| N_{0,u} =N  , \;  \hat{\mu}_{\{0\},u}(\widehat{V}_+^M) = N \big] >0.
\end{equation}
On the event $F(u,N,M)= \{ N_{0,u} =N  , \;  \hat{\mu}_{\{0\},u}(\widehat{V}_+^M) = N\}$, we denote by $Z_i$, $1\leq i \leq N$ the forward part of the trajectories with label at most $u$ hitting $0$, which are distributed  under $\mathbb{P}[ \;\! \cdot \;\! | F(u,N,M)]$ as $N$ independent (continuous-time) simple random walks starting at $0$, conditioned on visiting the origin at most $M$ times each. We consider the random variables 
\begin{equation*}
\zeta_n(Z_i) =
\begin{cases}
\sigma_{H_0^n(Z_i)}(Z_i) & \text{if }H_0^{n}(Z_i) < \infty, \\
0 & \text{else},
\end{cases}
 \quad \text{for $1\leq i \leq N$ and $1 \leq n \leq M$},
\end{equation*}
under $\mathbb{P}[ \;\! \cdot \;\! | F(u,N,M)]$. Note that, conditionally on $F(u,N,M)$, for each realization of the event $A_{u,\alpha}$, the origin is hit at most $N \cdot M$ times ``in total,'' i.e. by any trajectory in the interlacement at level $u$. Hence, noting that $L_{0,u} = \sum_{1\leq i \leq N} \sum_{1 \leq n \leq M} \zeta_n(Z_i)$ under $\mathbb{P}[ \;\! \cdot \;\! | F(u,N,M)]$, we obtain
\begin{equation}\label{FE1proof5}
\begin{split}
&\mathbb{P}\big[A_{u,\alpha}, \;  L_{0,u} \leq \alpha \big| F(u,N,M) \big] \\
&\qquad \geq \mathbb{P}\big[A_{u,\alpha}, \;  \zeta_n(Z_i) \leq \alpha / MN, \; 1\leq i \leq N, \; 1 \leq n \leq M \big| F(u,N,M) \big]   \\
&\qquad \geq \big( 1- e^{- \alpha / MN} \big)^{MN} \cdot \mathbb{P}\big[A_{u,\alpha},  \big| F(u,N,M) \big] \\
&\qquad > 0, 
\end{split}
\end{equation}
where we used in the third line that the variables $\zeta_n(Z_i)$, $1\leq i \leq N$, $1 \leq n \leq M$, are independent of $A_{u,\alpha}$ under $\mathbb{P}[ \;\! \cdot \;\! | F(u,N,M)]$, and \eqref{FE1proof4} in the last step. Finally, \eqref{FE1proof5} implies $\mathbb{P}[A_{u,\alpha}, L_{0,u} \leq \alpha] >0$, which is \eqref{T:FE11}. This completes the proof of \eqref{T:FE1}, and thus of Theorem \ref{T:FE}.
\end{proof}

\begin{remark}\label{FE_V^u,alpha}
Equivalently to the statement of Theorem \ref{T:FE}, the finite energy property holds for the measure $\widetilde{Q}_{u,\alpha}= (1\{x \in \mathcal{V}_{u,\alpha} \})_{x \in \mathbb{Z}^d}$,  for all $u,\alpha > 0$. Indeed, $\widetilde{Q}_{u,\alpha}$ is translation invariant, see Remark \ref{R:0-1law}, and given some $A \in \sigma(Y_z, z \neq 0 )$ satisfying $\widetilde{Q}_{u,\alpha}[A]>0$,
\begin{equation*}
\widetilde{Q}_{u,\alpha}[A \cap \{ Y_0 = 1\}] = Q_{u,\alpha}[\; \overline{A} \cap \{ Y_0 = 0\}] \stackrel{\eqref{T:FE11}}{>}0,
\end{equation*}
where $\overline{A}$ denotes the ``flipped'' event (i.e. $\overline{A}= \iota(A)$ with the inversion map $\iota:  \{0,1 \}^{\mathbb{Z}^d} \rightarrow \{0,1 \}^{\mathbb{Z}^d}$ such that $Y_x \circ \iota = 1- Y_x$, for all $x \in \mathbb{Z}^d$). Similarly, $\widetilde{Q}_{u,\alpha}[A \cap \{ Y_0 = 0\}]>0$ follows from \eqref{T:FE12}. \hfill $\square$
\end{remark}

\begin{corollary} \label{C:UNIQUENESS}  For any $u, \alpha \geq 0$ such that $\eta^{\mathcal{I}}(u,\alpha)>0$ $($see \eqref{eta^I}$)$, respectively $\eta^{\mathcal{V}}(u,\alpha)>0$, the set $\mathcal{I}^{u,\alpha}$, respectively $\mathcal{V}^{u,\alpha}$, contains almost surely a unique infinite component.
\end{corollary}

\begin{proof}
The set $\mathcal{I}^{0,\alpha}$ is (almost surely) empty for all $\alpha \geq 0$. For $u >0$, $\mathcal{I}^{u,0} = \mathcal{I}^{u}$ is almost surely an infinite connected subset of $\mathbb{Z}^d$, cf. \cite{S1}, Corollary 2.3. For $u >0$ and $\alpha > 0$ with $\eta^{\mathcal{I}}(u,\alpha)>0$, the claim follows from a theorem by Burton and Keane (see \cite{BK}, Theorem 2,   or \cite{HJ}, Theorem 12.2), since $Q_{u,\alpha}$ is translation invariant, see Lemma \ref{trans_inv}, and has the finite energy property by Theorem \ref{T:FE}. The same argument applies to $\mathcal{V}^{u,\alpha}$, for all $u,\alpha >0$ such that $\eta^{\mathcal{V}}(u,\alpha)>0$, by virtue of Remark \ref{FE_V^u,alpha}. Finally, uniqueness of the infinite cluster of $\mathcal{V}^{u}=\mathcal{V}^{u,0}$ whenever $\eta^{\mathcal{V}}(u,0)>0$ has been shown in \cite{T1}.
\end{proof}

\section{Percolation of $\mathcal{I}^{u,\alpha}$ for $u>0$ and small positive $\alpha$} \label{I_PERC}

In this section, we show that for any fixed value of $u>0$, the random set $\mathcal{I}^{u, \alpha}$ defined in \eqref{I^u,alpha} contains an infinite connected component with probability one for all sufficiently small, \textit{positive} values of $\alpha$ (depending on $u$), i.e. that the critical level $\alpha_*(u)$, see \eqref{alpha_*}, is strictly positive for all $u>0$. Note that by construction, $\mathcal{I}^{u, 0}$ percolates for every $u >0$.

\begin{theorem} \label{T:ALPHA_*PERC}
For all $u>0$,
\begin{equation} \label{T:U_*perc1}
 \alpha_*(u) >0.
\end{equation}
Moreover, given $u>0$, there exists $R\geq1$ such that for all sufficiently small $\alpha >0$, $\mathcal{I}^{u,\alpha}$ contains $\mathbb{P}$-almost surely an infinite component in the slab $\mathbb{Z}^2 \times [0,R)^{d-2}$.
\end{theorem}

\begin{proof}
 We aim at showing that the random set $\mathcal{I}^{u, \alpha}$ stochastically dominates a certain Bernoulli percolation on the interlacement $\mathcal{I}^{u}$. To this end, we define, for every $u \geq 0$, a collection of random variables $(\sigma_{x,u})_{x\in \mathbb{Z}^d}$ on $(\Omega,\mathcal{A}, \mathbb{P})$ (see below \eqref{intensity}) as follows. We recall that $\omega_{\{ x\}, u}$ denotes the restriction of the cloud $\omega \in \Omega$ to the trajectories visiting $x$ with label at most $u$, see \eqref{omega_K,u}, and introduce the $\mathbb{P}$-a.s. well-defined
\begin{equation} \label{T:U_*perc_pf1}
\begin{split}
 &\sigma_{x,u}(\omega) =
 \begin{cases}
 \sigma_{H_x(\widehat{w}_1)}(\widehat{w}_1), &  x \in \mathcal{I}^u(\omega) \\
 0, & \text{else}
 \end{cases} 
 , \\ 
 &\text{if } \omega = \omega_{\{ x\}, u} + \big( \omega - \omega_{\{ x\}, u} \big) \text{ with }  \omega_{\{ x\}, u} = \sum_{i=1}^n \delta_{\widehat{w}_i^*, u_i},
 \end{split}
 \end{equation}
for $x \in \mathbb{Z}^d$ and $u \geq 0$, where we assume for definiteness that the trajectories in the support of $\omega_{\{ x\}, u}$ are ordered according to increasing labels $u_i$ (which are $\mathbb{P}$-a.s. different) and $\widehat{w}_i$ stands for an arbitrary element in the equivalence class of $\widehat{w}_i^*$, i.e. $\pi^*(\widehat{w}_i) =\widehat{w}_i^*$, for all $ 1 \leq i \leq N$. Thus, $\sigma_{x,u}$ vanishes if $x$ lies in the vacant set $\mathcal{V}^u = \mathbb{Z}^d \setminus \mathcal{I}^u$ at level $u$, and otherwise collects the first-passage holding time at $x$ of the interlacement trajectory with smallest label passing through $x$. To fix ideas, if $\widetilde{\Omega}_u \subset \Omega$ denotes the subset of full $\mathbb{P}$-measure on which the collection $(\sigma_{x,u})_{x\in \mathbb{Z}^d}$ is well-defined, we set $\sigma_{x,u} (\omega)=0$ for all $\omega \in \Omega \setminus \widetilde{\Omega}_u$ and $x \in \mathbb{Z}^d$.

By \eqref{T:U_*perc_pf1} and the definition \eqref{L} of local times, $ L_{x,u}(\omega) \geq \sigma_{x,u}(\omega)$, for all $x\in \mathbb{Z}^d$, $u>0$ and $\omega \in \Omega$, whence
\begin{equation} \label{T:U_*perc_pf2}
\begin{split}
 &\text{$Q_{u,\alpha}$, the law of $\big(1\{ x\in \mathcal{I}^{u,\alpha} \} \big)_{x\in\mathbb{Z}^d}$ under $\mathbb{P}$, stochastically dominates} \\
 &\text{the law of $\big(1\{ \sigma_{x,u} > \alpha \} \big)_{x\in\mathbb{Z}^d}$ under $\mathbb{P}$, for all $u > 0$, $\alpha \geq 0$}.
\end{split}
\end{equation}
Abbreviating $\chi_{x,u} = 1\{ x \in \mathcal{I}^u \}$, for $x\in \mathbb{Z}^d$, $u > 0$, and denoting by $\mathcal{F}_u$ the $\sigma$-algebra generated by the random variables $\chi_{x,u}$, $x\in\mathbb{Z}^d$, we will show that for any $B \in \mathcal{B}\big( \mathbb{R}_+^{\mathbb{Z}^d} \big)$ and $u >0$,
\begin{equation} \label{T:U_*perc_pf3}
 \mathbb{P} \big[ (\sigma_{x,u})_{x\in \mathbb{Z}^d} \in B  \big| \mathcal{F}_u \big] = \widetilde{P} \big[ (\chi_{x,u} \cdot \widetilde{\tau}_x)_{x \in \mathbb{Z}^d} \in B \big], \quad \text{$\mathbb{P}$-a.s},
\end{equation}
where $\widetilde{P}$ is an auxiliary probability (which does not act on $\chi_{x,u}$, $x \in \mathbb{Z}^d$) under which $\widetilde{\tau}_x$, $x\in \mathbb{Z}^d$, are distributed as independent exponential variables with parameter one.

Before turning to the proof of \eqref{T:U_*perc_pf3}, we first explain how the assertion \eqref{T:U_*perc1} follows. Letting $\mathcal{J}^{u,\alpha}(\omega)= \{ x\in \mathbb{Z}^d ; \sigma_{x,u}(\omega)> \alpha \}$ $(\subseteq \mathcal{I}^{u,\alpha}(\omega))$, we have, for all $u >0$ and $\alpha \geq 0$,
\begin{equation} \label{T:U_*perc_pf4}
\begin{array}{rcl}
 \mathbb{P} \big[ 0 \stackrel{\mathcal{J}^{u,\alpha}}{\longleftrightarrow} \infty \big] & \hspace{-1ex} = & \hspace{-1ex} \mathbb{E} \big[ \;\! \mathbb{P} [ 0 \stackrel{\mathcal{J}^{u,\alpha}}{\longleftrightarrow} \infty | \mathcal{F}_u ]  \big] \\
& \hspace{-1ex} \stackrel{\eqref{T:U_*perc_pf3}}{=} & \hspace{-1ex} \mathbb{E} \big[ \!\; \widetilde{P} [\text{0 lies in an infinite cluster of $\{x \in \mathbb{Z}^d; \chi_{x,u} \cdot \widetilde{\tau}_x > \alpha \}$}] \big] \\
& \hspace{-1ex} = & \hspace{-1ex} \mathbb{P} \otimes \widetilde{P} \big[ 0 \stackrel{\mathcal{I}^{u} \cap \widetilde{\mathcal{B}}_\alpha}{\longleftrightarrow}  \infty \big],
\end{array}
\end{equation}
where we have defined $\widetilde{\mathcal{B}}_\alpha = \{ x\in \mathbb{Z}^d \;  ; \widetilde{\tau}_x > \alpha \}$ and used in the last step that $\chi_{x,u} \cdot \widetilde{\tau}_x > \alpha$ if and only if $x\in \mathcal{I}^u$ and $\widetilde{\tau}_x> \alpha$, for all $u>0$, $\alpha \geq 0$ and $x\in \mathbb{Z}^d$. The problem of Bernoulli site percolation on the interlacement set has been recently studied in \cite{RaSa}. By virtue of Theorem 1 therein, it follows together with \eqref{T:U_*perc_pf4} that for every $u >0$, there exists $\alpha_0 = \alpha_0 (u) > 0$ such that $\mathcal{J}^{u,\alpha_0}$ percolates in a thick two-dimensional slab. In particular, it follows that $\mathbb{P} \big[ 0 \stackrel{\mathcal{J}^{u,\alpha_0}}{\longleftrightarrow} \infty \big] > 0$, and in turn from  \eqref{T:U_*perc_pf2} that $\mathbb{P} \big[ 0 \stackrel{\mathcal{I}^{u,\alpha_0}}{\longleftrightarrow} \infty \big]>0$, i.e. $\alpha_*(u) \geq \alpha_0(u) > 0$. This completes the proof of \eqref{T:U_*perc1}.

It remains to prove \eqref{T:U_*perc_pf3}. We begin with an elementary result on randomly indexed random variables, which is tailored to our purposes. The setting is as follows. Let $\mathbb{M}$ be a countable (indexing) set, $\tau$ a real-valued random variable with law $P_{ \tau}$ and $U \subset \subset \mathbb{Z}^d$. Suppose that, under some probability $P$,  $(\tau_m)_{m \in \mathbb{M}}$ is a family of i.i.d. copies of $\tau$ and $\mu = (\mu(x))_{x\in U}$ is a random element of $\mathbb{M}^U$ with the property that $\mu(x) \neq \mu(y)$ whenever $x \neq y$. Moreover, $\mu$ be measurable with respect to some $\sigma$-algebra $\mathcal{G}$ independent of $\sigma(\tau_m, \, m \in \mathbb{M})$.

\begin{lemma} \label{L:random_labels} For all bounded, measurable functions $f: \mathbb{R}^U \to \mathbb{R}$,
\begin{equation} \label{EQ:random_labels}
E[f(( \tau_{\mu(x)})_{x\in U}) \, | \mathcal{G} ] = E_{ \tau}^{\otimes U}[f], \quad \text{$E$-a.s.}
\end{equation}
\end{lemma}
\begin{proof}
By independence, for all bounded, real-valued, $\mathcal{G}$-measurable functions $g$,
\begin{equation*}
\begin{split}
E[f(( \tau_{\mu(x)})_{x\in U})  \cdot g  ] &= \sum_{(m_x)_{ x\in U}} E[f(( \tau_{m_x})_{x\in U})]  \cdot E[g 1\{\mu(x)=m_x, x\in U \} ] = E_{ \tau}^{\otimes U}[f] \cdot E[g], 
\end{split}
\end{equation*}
where the sum is over elements of $ \mathbb{M}^U$ (with $m_x \neq m_y$ for all $x \neq y$) and the second step follows because the variables $\tau_m$, $m \in \mathbb{M}$, are identically distributed.
\end{proof}

We now explain how \eqref{EQ:random_labels} can be applied in the present context of random interlacements to yield \eqref{T:U_*perc_pf3}. By Dynkin's lemma, the joint law of $ (\sigma_{x,u}, \; \! \chi_{x,u})_{x\in \mathbb{Z}^d}$ is fully specified by the distribution of its finite-dimensional marginals, thus, in order  to prove \eqref{T:U_*perc_pf3}, it suffices to show
\begin{equation} \label{T:U_*perc_pf6}
\mathbb{P} [\sigma_{x,u} \in B_{x},  \; \chi_{x,u} = \varepsilon_{x}, \; x \in K ]  =  \mathbb{P} \otimes \widetilde{P} [\chi_{x,u} \cdot \widetilde{\tau}_{x} \in B_{x}, \;  \chi_{x,u} = \varepsilon_{x}, \; x \in K],
\end{equation}
for all $K \subset \subset \mathbb{Z}^d$, $u>0$, $B_{x}\in \mathcal{B}(\mathbb{R}_+)$ and $\varepsilon_{x} \in \{ 0,1 \}$, $x \in K$. Defining $U =   \{ x \in K ; \; \! \varepsilon_{x}=1 \}$, we may assume that $B_{x} \ni 0 $ for all $x \in K \setminus U$ in \eqref{T:U_*perc_pf6}, for otherwise both sides of \eqref{T:U_*perc_pf6} vanish ($\chi_{x,u} = 0$ means $x \in \mathcal{V}^u$, whence $\sigma_{x,u}=0$ by \eqref{T:U_*perc_pf1}). Moreover, we may condition on $N_{K,u}$, the number of trajectories with label at most $u$ hitting $K$ to be a fixed number $N \geq 0$, since the general case follows by summation over $N$. All in all, letting $\mathbb{P}_N= \mathbb{P}[\, \cdot \, | N_{K,u} = N]$, it remains to show
\begin{equation} \label{T:U_*perc_pf8}
\begin{split}
  &\mathbb{P}_N [\sigma_{x,u} \in B_{x}, \; x \in U, \; \chi_{x,u} = \varepsilon_{x}, \; x \in K] \\
  & \qquad \qquad =  \widetilde{P}  [ \widetilde{\tau}_{x} \in B_{x}, \; x \in U] \cdot \mathbb{P}_N[   \chi_{x,u} = \varepsilon_{x}, \; x \in K],
\end{split}
\end{equation}
for all $N \geq 0$, $K \subset \subset \mathbb{Z}^d$, $u>0$, $\varepsilon_{x} \in \{ 0,1 \}$, $x \in K$, and $B_{x}\in \mathcal{B}(\mathbb{R}_+)$, $x \in U \, ( =   \{ x \in K ; \; \! \varepsilon_{x}=1 \})$, which we all assume to be fixed from now on. We suppose that $U \neq \emptyset$ and $N > 0$, since \eqref{T:U_*perc_pf8} trivially holds otherwise. Observe that the event on the left hand side of \eqref{T:U_*perc_pf8} only depends on $\omega$ ``through'' $\omega_{K,u}$, the restriction of $\omega$ to all trajectories with label at most $u$ hitting $K$ (recall \eqref{omega_K,u}). Now, under $\mathbb{P}_N$ (see for example \cite{R}, p. 132),
\begin{equation} \label{T:U_*perc_pf9}
 \omega_{K,u} \stackrel{\text{law}}{=} \sum_{i=1}^{N} \delta_{(\pi^*(\widehat{Z}_i),v_i)},
\end{equation}
where $\widehat{Z}_i$, $1\leq i \leq N$, are independent, $\widehat{W}$-valued random elements all distributed according to $ \widehat{Q}_K (d\widehat{w})/\text{cap}(K)$, so they are at time $0$ for the first time in $K$, and $v_1< \cdots < v_{N}$ are independent of $\widehat{Z}_1,\dots, \widehat{Z}_{N}$ and obtained by reordering independent uniform random variables on $[0,u]$. We denote by $\mathbb{P}'_N$ the probability governing these auxiliary random variables.

We now apply Lemma \ref{L:random_labels} to the left-hand side of \eqref{T:U_*perc_pf8} using the representation of $\omega_{K,u}$ (under $\mathbb{P}_N$) given by \eqref{T:U_*perc_pf9}. To this end, we let $\mathcal{G}$ be the $\sigma$-algebra generated by the random variables
\begin{equation*}
 \text{$v_i$, $1\leq i \leq N$, and $X(\widehat{Z}_i)= \big( X_n(\widehat{Z}_i )\big)_{n \geq0}$, $1\leq i \leq N$},
\end{equation*}
and observe that, in the above identity in law \eqref{T:U_*perc_pf9}, $(\chi_{x,u})_{x\in K}$ corresponds to
\begin{equation*}
(\chi_{x,u}')_{x\in K} \stackrel{\text{def.}}{=} \big( 1\{ x \in \text{range}(X(\widehat{Z}_i)) \text{ for some } 1\leq i \leq N\} \big)_{x\in K}.
\end{equation*}
In particular, $\chi_{x,u}'$, $x \in K$, are measurable with respect to $\mathcal{G}$. Moreover, by definition of $\widehat{Q}_K$ in \eqref{intensity} and the discussion following \eqref{T:U_*perc_pf9}, the random variables 
\begin{equation*}
\tau_m \stackrel{\text{def.}}{=} \sigma_{n}(\widehat{Z}_i), \quad \text{for $m =(i,n) \in \{1,\dots,N\} \times \mathbb{N} \stackrel{\text{def.}}{=} \mathbb{M}$,}
\end{equation*}
are i.i.d. exponential variables with parameter $1$, independent of $\mathcal{G}$. Finally, for all $x \in U$, introducing $\mu(x) = (\iota(x), \nu(x))$ with
$$
\iota(x)=  \min \{ 1\leq i \leq N ; \, x \in \text{range}(X(\widehat{Z}_i)) \}, \qquad  \nu(x)= H_{\{ x \}}(\widehat{Z}_{\iota(x)})
$$
(both $\mathcal{G}$-measurable), we see that the variable $\sigma_{x,u}$ defined in \eqref{T:U_*perc_pf1} corresponds to $\tau_{\mu(x)}$ in the representation \eqref{T:U_*perc_pf9}, and by construction, $\mu(x)\neq \mu(y)$ whenever $x \neq y$. All in all, we obtain that the left-hand-side of \eqref{T:U_*perc_pf8} can be rewritten as
\begin{equation*}
\begin{split}
 &\mathbb{P}_N' [\tau_{\mu(x)} \in B_{x}, \; x \in U, \; \chi'_{x,u} = \varepsilon_{x}, \; x \in K] \\
 &\qquad \qquad = \mathbb{P}_N' [\, \mathbb{P}_N'[\tau_{\mu(x)} \in B_{x}, \; x \in U \, | \, \mathcal{G}], \, \chi'_{x,u} = \varepsilon_{x}, \; x \in K],
\end{split}
\end{equation*}
and \eqref{T:U_*perc_pf8} follows immediately by virtue of \eqref{EQ:random_labels}. This completes the proof of \eqref{T:U_*perc_pf3}, and thus of Theorem \ref{T:ALPHA_*PERC}.
\end{proof}

\section{Results on the Gaussian free field} \label{GFF_NOPERC}

We now turn to our other main object of study, the Gaussian free field on $\mathbb{Z}^d$, $d\geq3$, as defined in \eqref{GFF}, and prove in Theorem \ref{T:GFF_NOPERC} below that the level sets $\mathcal{L}^{\geq h}$ defined in \eqref{levelsets} do not percolate when $h \geq 0$ is sufficiently large. We will use this fact in Section \ref{RI_NOPERC} as a crucial preliminary step towards addressing the issue of absence of percolation for the sets $\mathcal{I}^{u,\alpha}$ and $\mathcal{V}^{u,\alpha}$. To begin with, we adapt some of the results obtained in \cite{RS} to our present purposes. This includes setting up an appropriate renormalization scheme. Thus, we introduce a geometrically increasing sequence of length scales
\begin{equation} \label{L_n}
L_{n} = l_0^n L_0, \quad \text{for } n \geq 0,
\end{equation}
with $L_0 \geq 1$, $l_0 \geq 100$ to be specified below, and corresponding renormalized lattices
\begin{equation} \label{LL_n}
\mathbb{L}_n = L_n \mathbb{Z}^d, \quad \text{for } n \geq 0.
\end{equation}
We also define 
\begin{equation} \label{B_n,x}
B_{n,x} = x + \big([0,L_n)\cap \mathbb{Z} \big)^d, \quad \text{for $n \geq 0$ and $x \in \mathbb{L}_n$},
\end{equation}
so that $\{ B_{n,x}  ; \; x \in \mathbb{L}_n \}$ defines a partition of $\mathbb{Z}^d$ into boxes of side length $L_n$ for all $n \geq 0$. For any vertex $x\in \mathbb{Z}^d$, we denote by $y_0(x)$ the unique vertex in $\mathbb{L}_0$ such that $x\in B_{0,y_0(x)}$. 

A finite sequence $\pi =(x_i)_{0\leq i \leq N}$ with $0\leq N < \infty$, $x_i \in \mathbb{L}_n$ for all $0 \leq i \leq N$  and $|x_{i+1} -x_i|= L_n$ for all $0 \leq i \leq N-1$ will be called a nearest-neighbor path (of length $N$) in $\mathbb{L}_n$, for all $n\geq 0$. An (infinite) nearest-neighbor path $\pi =(x_i)_{i\geq 0}$ in $\mathbb{L}_n$ is defined similarly.

Given some nearest-neighbor path $\pi = (x_i)_{i\geq 0}$ in $\mathbb{Z}^d$, we consider the (sub-)sequence $(x_{i_k})_{ 0 \leq k \leq M}$, where $i_0=0$, $i_{k+1}= \inf \{ i > i_k \; ; \; x_i \not \in B_{0, y_0(x_{i_k})}\}$, for all $k\geq 0$ (with the convention $\inf \emptyset = \infty$), and $M = \sup \{ k \geq 0 \; ; \; i_k < \infty \} \; (\leq \infty)$. We then define the \textit{trace} of $\pi$ on $\mathbb{L}_0$ as the sequence
\begin{equation} \label{trace}
 \pi_0 = \big( y_0(x_{i_k}) \big)_{0 \leq k \leq M}.
\end{equation}
It follows that
\begin{equation} \label{trace_path_nn}
 \text{$\pi_0$ is a nearest-neighbor path in $\mathbb{L}_0$ (of possibly finite length)}.
\end{equation}
The same construction works if $\pi$ itself has only finite length, say $N$ (one then defines $i_0=0$, $i_{k+1}= \inf \{ i \in \{ i_k+1,\dots,N \} \; ; \; x_i \not \in B_{0, y_0(x_{i_k})} \}$). In this case, $\pi_0$ necessarily has finite length, too, i.e. $M < \infty$.

We write $T^{(k)} = \{1,2 \}^k$ for all $k \geq 0$ (with the convention $\{1,2 \}^0 = \emptyset$), and $T_n = \bigcup_{0 \leq k \leq n}T^{(k)}$ for the canonical dyadic tree of depth $n$. For any given parameter
\begin{equation} \label{rparameter}
r \geq 10 \text{ satisfying } 2r \leq l_0,
\end{equation}
we call a map $\mathcal{T}: T_n \rightarrow \mathbb{Z}^d$ a \textit{proper embedding of $T_n$ in $\mathbb{Z}^d$ with root at $x \in \mathbb{L}_n$} if
\begin{equation} \label{embedding}
\begin{array}{ll}
\text{i)} & \mathcal{T}(\emptyset) =x, \\
\text{ii)} & \text{for all $ 0 \leq k < n$: if $m_1,m_2 \in T^{(k+1)}$ are the two descendants} \\
& \text{of $m \in T^{(k)}$, then $\mathcal{T}(m_1), \mathcal{T}(m_2) \in \mathbb{L}_{n-k-1} \cap B_{n-k,\mathcal{T}(m)}$, and} \\
& \text{$|\mathcal{T}(m_1)- \mathcal{T}(m_2)|_{\infty} > \frac{L_{n-k}}{r}$}.
\end{array}
\end{equation}
Note that, together with \eqref{L_n}, the condition $l_0 \geq 2r$ in \eqref{rparameter} guarantees that for any $m_1, m_2$ as above, the $\ell^\infty$-distance between the sets $B_{n-k-1,\mathcal{T}(m_1)}$ and $B_{n-k-1,\mathcal{T}(m_2)}$ is bounded from below by $L_{n-k-1}$ (see also Remark \ref{R:DEC_INEQ}, 1) below).

We denote by $\Lambda_{n,x}$ the set of proper embeddings of $T_n$ in $\mathbb{Z}^d$ with root at $x \in \mathbb{L}_n$, for $n \geq 0$. One easily infers that
\begin{equation} \label{Lambda_n,x}
|\Lambda_{n,x}| \leq (l_0^{d})^2 \cdot (l_0^{d})^{2^2} \cdots (l_0^{d})^{2^n} = (l_0^{d})^{2(2^n-1)} \leq (l_0^{2d})^{2^n}.
\end{equation}
We now describe the events of interest. For each $x\in \mathbb{L}_0$, let $A_{0,x}$ be a measurable subset of $\{0,1\}^{\mathbb{Z}^d}$ (endowed with its canonical $\sigma$-algebra, and with canonical coordinates $Y_z$, $z \in \mathbb{Z}^d$). The collection $\{A_{0,x}  ; \; x \in \mathbb{L}_0 \}$ is said to be $\mathbb{L}_0$\textit{-adapted} if
\begin{equation} \label{adapted_events}
 A_{0,x} \in \sigma(Y_z  ; \; z \in B_{0,x}), \text{ for all $x \in\mathbb{L}_0$}.
\end{equation}
For arbitrary $h \in \mathbb{R}$, we introduce the measurable map 
\begin{equation} \label{phi_h}
\Phi_h : \mathbb{R}^{\mathbb{Z}^d} \rightarrow \{0,1\}^{\mathbb{Z}^d}, \quad (\varphi_x)_{x\in\mathbb{Z}^d} \mapsto \big( 1\{ \varphi_x \geq h \} \big)_{x\in\mathbb{Z}^d},
\end{equation}
and, given an $\mathbb{L}_0$-adapted collection $A= \{A_{0,x} ; \; x\in \mathbb{L}_0\}$, consider the quantity
\begin{equation} \label{p_n}
p^A_n(h)= \sup_{x\in \mathbb{L}_n , \; \mathcal{T} \in \Lambda_{n,x}} P^G \Big[ \bigcap_{m \in T^{(n)}} \Phi_h^{-1} \big( A_{0, \mathcal{T}(m)} \big) \Big], \text{ for $n\geq0$, $h\in\mathbb{R}$}
\end{equation}
(recall that $P^G$ denotes the law of Gaussian free field on $\mathbb{Z}^d$, see \eqref{GFF}). The following proposition provides ``recursive'' bounds for $p^A_n(h_n)$ along a suitable sequence $(h_n)_{n\geq 0}$, for certain collections $A$ of $\mathbb{L}_0$-adapted events. 

\begin{proposition} \label{P:DEC_INEQ} $(L_0 \geq 1$, $r \geq 10$, $l_0 \geq 100 \vee 2r$ in \eqref{L_n} and \eqref{rparameter}$)$

\medskip
\noindent There exist positive constants $c_1$ and $c_2$ such that, defining 
\begin{equation} \label{M}
M(n,L_0)= c_2 \big( \log(2^n L_0^d) \big)^{1/2},
\end{equation}
then, given any positive sequence $(\beta_n)_{n \geq 0}$ satisfying
\begin{equation} \label{beta_n_cond}
\beta_n \geq (\log 2)^{1/2} + M(n,L_0), \ \text{for all } n \geq 0,
\end{equation}
and any increasing, real-valued sequence $(h_n)_{n \geq 0}$ satisfying
\vspace{-0.5ex}
\begin{equation} \label{h_n_cond}
h_{n+1} \geq h_n + c_1 \beta_n r^{d-2} \big(2 l_0^{-(d-2)} \big)^{n+1}, \ \text{for all } n \geq 0,
\end{equation}
for all $\mathbb{L}_0$-adapted collections $A= \{ A_{0,x}; \; x \in \mathbb{L}_0\}$ of increasing events, respectively $A'= \{ A'_{0,x}; \; x \in \mathbb{L}_0\}$ of decreasing events, one has
\begin{equation} \label{dec_ineq_incr}
p^A_{n+1}(h_{n+1}) \leq p^A_n(h_n)^2 + 3  e^{-(\beta_n -M(n,L_0))^2}, \ \text{for all } n \geq 0,
\end{equation}
respectively
\begin{equation} \label{dec_ineq_decr}
p^{A'}_{n+1}(-h_{n+1}) \leq p^{A'}_n(-h_n)^2 + 3  e^{-(\beta_n -M(n,L_0))^2}, \ \text{for all } n \geq 0.
\end{equation}
\end{proposition}

\begin{remark}$\quad$ \label{R:DEC_INEQ} 1) To prove Proposition \ref{P:DEC_INEQ}, one simply follows the steps of the proof of Proposition 2.2 in \cite{RS} (see also Remark 2.3, 2) in \cite{RS}), with minor modifications: denoting by $T_i^{(n+1)}$, $i=1,2$, the set of leaves in $T_{n+1}$ which are descendants of vertex $i$, i.e. $T_i^{(n+1)}= \{ (i,i_2,\dots,i_{n+1}) \; ; \; i_k\in \{1,2\}, \; 2 \leq k \leq n+1\}$, and letting $K_{\mathcal{T},i} = \bigcup_{m \in T_i^{(n+1)}} B_{0, \mathcal{T}(m)}$, for $\mathcal{T} \in \Lambda_{n+1,x}$, $x \in \mathbb{L}_{n+1}$, one has, by virtue of \eqref{embedding} ii), $K_{\mathcal{T},i} \subset B_{n, \mathcal{T}(i)}$, for $i=1,2$. Hence,
\begin{align*}
 d(K_{\mathcal{T},1}, K_{\mathcal{T},2}) &\geq d\big(B_{n, \mathcal{T}(1)},B_{n, \mathcal{T}(2)}\big) \\ & = |\mathcal{T}(1)-\mathcal{T}(2)|_\infty - L_n \stackrel{\eqref{embedding}}{ >} \frac{L_{n+1}}{r} -  L_n \stackrel{\eqref{L_n},\eqref{rparameter}}{\geq} c \cdot \frac{L_{n+1}}{r},
\end{align*}
for some $c>0$ and all $\mathcal{T} \in \Lambda_{n+1,x}$, $n\geq 0$, $x \in \mathbb{L}_{n+1}$. In contrast, to deduce (2.31) in \cite{RS}, one uses the fact that the sets $K_1$, $K_2$ defined in (2.27) of \cite{RS} satisfy $d(K_1,K_2)\geq c'L_{n+1}$. This accounts for the change from (2.23) in \cite{RS} to \eqref{h_n_cond} above.

Furthermore, by \eqref{adapted_events}, the event on the right-hand side of \eqref{p_n} is measurable with respect to the $\sigma$-algebra generated by the coordinates $\varphi_x$, for $x \in  \bigcup_{m \in T^{(n)}}B_{0,\mathcal{T}(m)}$. The cardinality of this set is bounded by $2^nL_0^d$, which justifies the modification in the definition of $M(n,L_0)$ from (2.21) in \cite{RS} to \eqref{M} above.

\medskip
\noindent 2) In all applications below, the collections of $\mathbb{L}_0$-adapted events will be of the form $A=\{ t_x(A_{0,0}) ; \; x\in \mathbb{L}_0 \}$ ($(t_x)_{x\in \mathbb{Z}^d}$ denote the canonical shifts on $\{0,1\}^{\mathbb{Z}^d}$), with $A_{0,0}$ increasing or decreasing and measurable with respect to the coordinates in $B_{0,0}$. In this case, in the definition \eqref{p_n} of $p_n^A(h)$, it suffices to take the supremum over all $\mathcal{T}\in \Lambda_{n,x}$ for an arbitrary fixed $x\in \mathbb{L}_n$, by translation invariance.

\medskip
\noindent 3) By definition, for all $\mathbb{L}_0$-adapted collections $A$, $A'$, of increasing, respectively decreasing, events, and for all $n\geq 0$,
\begin{equation} \label{p_n_monotone}
 \text{$p_n^A(h)$ is a non-increasing and $p_n^{A'}(h)$ a non-decreasing function of $h\in\mathbb{R}$}.
\end{equation}

\medskip
\noindent 4) For future reference, we define two particular collections of $\mathbb{L}_0$-adapted events, namely
\begin{equation} \label{A}
 A= \{ A_{0,x} \; ; \; x \in \mathbb{L}_0\}, \text{ with } A_{0,x} = \bigcup_{y \in B_{0,x}} \{ Y_y = 1 \},  \text{ and } A'= \{ \overline{A}_{0,x} \; ; \; x \in \mathbb{L}_0\}, 
\end{equation}
where $\overline{A}_{0,x} = \bigcup_{y \in B_{0,x}} \{ Y_y = 0 \}$ denotes the event obtained by ``flipping'' all configurations of $A_{0,x}$, cf. Remark \ref{FE_V^u,alpha}. By symmetry, one easily deduces that $p_n^A(h)= p_n^{A'}(-h)$, for $A,A'$ as defined in \eqref{A}, all $n\geq 0$ and $h \in \mathbb{R}$.

\medskip
\noindent 5) As is to be expected, the parameter $r$, which regulates the sparsity of the tree in \eqref{embedding}, competes against the ``sprinkling'' condition \eqref{h_n_cond} (intuitively, the bigger $r$ is, the closer the leaves of the tree can potentially get, hence the larger the increase in parameter $h_{n+1}-h_{n}$ needs to be in order to dominate the interactions). For the purposes of the present work, taking $r$ proportional to $l_0$ will suffice (thus, one could omit it completely from the picture). We do however keep track of the parameter $r$ in our presentation of the renormalization scheme in anticipation of future applications, for which it might be needed. \hfill$\square$ 
\end{remark}

Upon selecting (as in (2.51) of \cite{RS})
\begin{equation} \label{beta_n}
\beta_n = (\log 2)^{1/2} + M(n,L_0) + 2^{(n+1)/2} \big(n^{1/2} + K_0^{1/2}\big), \quad n \geq 0,
\end{equation}
for some $K_0 > 0$ to be specified below in \eqref{K_0}, and with $M(n,L_0)$ as defined in \eqref{M} (note that \eqref{beta_n} satisfies the condition \eqref{beta_n_cond} for every choice of $K_0 >0$), one can inductively propagate the bounds \eqref{dec_ineq_incr} and \eqref{dec_ineq_decr}, provided the induction can be initiated, see \eqref{p_0_cond} below.

\begin{proposition} \label{P:DEC_INEQ_PROPAGATED} $(L_0 \geq 1$, $r \geq 10$, $l_0 \geq 100 \vee 2r)$

\medskip
\noindent Let $A$ be an $\mathbb{L}_0$-adapted sequence of increasing events, respectively $A'$ an $\mathbb{L}_0$-adapted sequence of decreasing events. Assume $h_0 \in \mathbb{R}$ and $K_0 \geq 3(1-e^{-1})^{-1} \stackrel{\textrm{def.}}{=}B$ are such that
\begin{equation} \label{p_0_cond}
p_0^A(h_0) \leq e^{-K_0}, \quad \text{resp.} \quad p_0^{A'}(-h_0) \leq e^{-K_0},
\end{equation}
and let the sequence $(h_n)_{n \geq 0}$ satisfy \eqref{h_n_cond} with $(\beta_n)_{n\geq0}$ as defined in \eqref{beta_n}. Then,
\begin{equation} \label{p_n_bounds}
p_n^A(h_n) \leq e^{-(K_0 - B)2^n},  \quad \text{resp.} \quad p_n^{A'}(-h_n) \leq e^{-(K_0 - B)2^n} \qquad \text{for all } n \geq 0. 
\end{equation}
\end{proposition}
\begin{proof} 
On account of Proposition \ref{P:DEC_INEQ} and \eqref{beta_n}, the proof of Proposition \ref{P:DEC_INEQ_PROPAGATED} is the same as that of Proposition 2.4 in \cite{RS}. 
\end{proof}

For the present purposes, it will suffice to select
\begin{equation} \label{K_0}
 K_0=  \log \big(2l_0^{2d}\big)+ B 
\end{equation}
in the definition \eqref{beta_n} of $\beta_n$. Moreover, we will solely consider (increasing) sequences $(h_n)_{n \geq0}$ with
\begin{equation} \label{h_n}
h_0 > 0,  \qquad h_{n+1} - h_n = c_1 \beta_n r^{d-2} \big(2 l_0^{-(d-2)} \big)^{n+1}, \qquad \text{for all } n \geq 0,
\end{equation}
so that condition \eqref{h_n_cond} is satisfied. Note that $L_0,$ $l_0$, $r$ and $h_0$ are the only parameters which remain to be selected in order to fully specify the sequence $(h_n)_{n \geq 0}$. But for any choice of $L_0 \geq 1$, $r \geq 10$, $l_0 \geq 100 \vee 2r$ and $h_0 > 0$, the limit $h_{\infty}= \lim_{n\to \infty} h_n$ is finite. Indeed, we observe that $\beta_n$ as defined in \eqref{beta_n} (with $M(n,L_0)$ given by \eqref{M} and $K_0$ by \eqref{K_0}) satisfies $\beta_n \leq c(L_0,l_0)2^{n+1}$, for all $n \geq 0$. Hence, 
\begin{equation} \label{h_infty}
h_\infty \stackrel{\eqref{h_n}}{=} h_0 +  c_1 r^{d-2} \sum_{n=0}^\infty \beta_n \big(2 l_0^{-(d-2)} \big)^{n+1} \leq h_0 + c'(L_0,l_0,r) \sum _{n=0}^\infty \big(4 l_0^{-(d-2)} \big)^{n+1} < \infty.
\end{equation}
This completes the description of the renormalization scheme. 

\bigskip

Next, we establish a (somewhat general) geometric lemma, similar to Lemma 6 of \cite{RaSa}, which will be useful in several instances below. The setting is as follows. Given an integer $N\geq1$ and a sequence of length scales $(L_n)_{n\geq 0}$ satisfying \eqref{L_n}, for some $L_0 \geq 1$, $r \geq 10$ and $l_0 \geq 100 \vee 2r$, we consider a family of events $A_{0,x}^{(i)}$, $x \in \mathbb{L}_0$, $1 \leq i \leq N$, on the space $\{0,1\}^{\mathbb{Z}^d}$ (equipped with its canonical $\sigma$--algebra), such that $A^{(i)} = \{A_{0,x}^{(i)} ; \; x \in \mathbb{L}_0 \}$ forms a collection of $\mathbb{L}_0$-adapted events, for all $i=1 \dots,N$. The reason for considering $N$ such collections (rather than just one) is the monotonicity condition on the events in Propositions \ref{P:DEC_INEQ} and \ref{P:DEC_INEQ_PROPAGATED}. In applications, $N$ will typically be less than $10$. 
 
In order to gain some control over the distance between the localized, but ``amorphous'' events in $A^{(i)}$, we impose certain geometric constraints, using the tree structure in \eqref{embedding}. To this end, we first recall that $\Lambda_{n,x}$ is the set of proper embeddings of $T_n$ (the canonical binary tree of depth $n$) in $\mathbb{Z}^d$ with root at $x\in\mathbb{L}_n$, see \eqref{embedding}, and define the events
\begin{equation}\label{A_n,x}
A_{n,x}^{(i)} = \bigcup_{\mathcal{T}\in \Lambda_{n,x}} \bigcap_{m\in T^{(n)}} A_{0,\mathcal{T}(m)}^{(i)}, \text{ for $n\geq 0$, $x\in\mathbb{L}_n$ and $1\leq i \leq N$}
\end{equation}
(this definition is consistent in the case $n=0$). In words, $A_{n,x}^{(i)}$ is the event that $2^n$ ``well-separated'' events $A_{0,y}^{(i)}$ (indexed by the leaves of a binary tree of depth $n$), ``located'' within the box $B_{n,x}$, all simultaneously occur. In particular, on account of \eqref{embedding}, if $A_{0,y}^{(i)}$ and $A_{0,y'}^{(i)}$ are any two such events, then $y,y' \in B_{n,x} \cap \mathbb{L}_0$ and $|y-y'|_\infty > L_1 / r \ (\geq 2L_0)$. Moreover, by \eqref{embedding}, we observe that any proper embedding $\mathcal{T}\in \Lambda_{n,x}$, $n\geq 1$, is uniquely determined by specifying $\mathcal{T}(i)=x_i \in \mathbb{L}_{n-1}\cap B_{n,x}$, $i=1,2$, with $|x_1-x_2|_\infty > L_{n}/r$, and $\mathcal{T}_i \in \Lambda_{n-1,x_i}$, for $i=1,2$, where $\mathcal{T}_i((i_1,\dots,i_k))= \mathcal{T}((i,i_1,\dots,i_k)),$ for $(i_1,\dots,i_k)\in T^{(k)}$, $0\leq k \leq n-1$, i.e., $\mathcal{T}_i$ is the embedding corresponding to the restriction of $\mathcal{T}$ to the descendants of $i$ in $T^n$, for $i=1,2$. Thus, from \eqref{A_n,x}, we obtain the recursion
\begin{equation}\label{A_n,xRECURSION}
A_{n,x}^{(i)} = \bigcup_{\substack{x_1,\; x_2 \in \mathbb{L}_{n-1}\cap B_{n,x} \\ |x_1-x_2|_\infty > L_{n}/r}} A_{n-1,x_1}^{(i)} \cap A_{n-1,x_2}^{(i)}, \text{ for $n\geq 1$, $x\in\mathbb{L}_n$ and $1\leq i \leq N$}.
\end{equation}

We now introduce some randomness into the setting, and assume to this end that $(\Omega, \mathcal{F},P)$ is a probability space on which $N$ $\{0,1\}^{\mathbb{Z}^d}$-valued random fields $\zeta^{(i)}=(\zeta^{(i)}_x)_{x \in \mathbb{Z}^d}$, $1 \leq i\leq N$, are defined. (In applications below, $\zeta^{(i)}$ will involve the Gaussian free field on $\mathbb{Z}^d$, and later also $(L_{x,u})_{x \in \mathbb{Z}^d}$, $u \geq 0$, the field of occupation times for continuous-time interlacements, see above \eqref{T:GFF_NOPERCpf11} and \eqref{T:U_*pf6}.) We consider the events
\begin{equation} \label{bad_events_general}
\mathcal{B}_{n,x}^{(i)} = \{ \zeta^{(i)} \in A_{n,x}^{(i)} \} 
\end{equation}
on $\Omega$, for $n\geq 0$, $x \in \mathbb{L}_n$ and $1 \leq i \leq N$, with $A_{n,x}^{(i)}$ as above. We think of the events $\mathcal{B}_{n,x}^{(i)}$ as ``bad'' events, and their probability will be typically very small, see \eqref{cram_condition} below. Accordingly, we define a vertex $x\in \mathbb{L}_n$ to be \textit{$n$-bad of type $i$} if the event $\mathcal{B}_{n,x}^{(i)}$ occurs (under $P$), and simply \textit{$n$-bad} if the type is not specified, i.e. if
\begin{equation*}
\bigcup_{i=1}^N \mathcal{B}_{n,x}^{(i)}
\end{equation*}
occurs. A $0$-bad vertex will be called \textit{bad}. Further, we define $\big\{B(x,L_n)\stackrel{\text{bad}}{\longleftrightarrow}S(x,2L_n)\big\}$, for $n\geq0$ and $x\in \mathbb{L}_n$, as the event that $B(x,L_n)\cap \mathbb{L}_0$ is connected to $S(x,2L_n)\cap \mathbb{L}_0$ by a nearest-neighbor path of bad vertices in $\mathbb{L}_0$. The following geometric lemma will be useful in proving that long paths of bad vertices in $\mathbb{L}_0$ have small probability. For future reference, let
\begin{equation} \label{c_3}
c_3(N)= 4(N4^d+1),  \text{ for $N \geq 1$}.
\end{equation}

\begin{lemma} \label{L:CRAMGEN} $(N\geq 1$, $L_0 \geq 1$, $r \geq c_3(N)$, and $l_0 \geq 2r)$

\medskip

\noindent  For all $n\geq0$ and $x\in \mathbb{L}_n$, one has
\begin{equation} \label{L:CRAMGEN_conclusion}
\big\{B(x,L_n)\stackrel{\text{bad}}{\longleftrightarrow}S(x,2L_n)\big\} \subseteq 
\bigcup_{y\in \mathbb{L}_n \cap (x+ [-2L_n,2L_n)^d)} \{ y \text{ is $n$-bad}\}.
\end{equation}
\end{lemma}

\begin{proof}
We proceed by induction over $n$. Clearly, \eqref{L:CRAMGEN_conclusion} holds for $n=0$ and arbitrary $x\in \mathbb{L}_0$. We now suppose it holds for $n-1$ and all $x\in \mathbb{L}_{n-1}$ (with $n \geq 1$). We consider an arbitrary vertex $x_0\in \mathbb{L}_n$ and assume that $\big\{ B(x_0,L_n)\stackrel{\text{bad}}{\longleftrightarrow}S(x_0,2L_n)\big\}$ occurs. To prove \eqref{L:CRAMGEN_conclusion}, we thus need to show that
\begin{equation} \label{L:CRAMproof1_1}
\text{there exists a vertex $y \in \mathbb{L}_n \cap \big( x_0 + [-2L_n,2L_n)^d \big)$ which is $n$-bad}.
\end{equation}
By definition, since the event $\big\{B(x_0,L_n)\stackrel{\text{bad}}{\longleftrightarrow}S(x_0,2L_n)\big\}$ occurs, there exists a nearest-neighbor path $\pi$ of bad vertices (in $\mathbb{L}_0$) connecting $B(x_0,L_n)$ to $S(x_0,2L_n)$. We claim that $\pi$ intersects the spheres $S(x_0, L_n+ 4[L_{n}/r]j)$, for all $j= 0,\dots, N 4^d$ (observe in passing that $[L_{n}/r] = [l_0 L_{n-1}/r] \geq 2L_{n-1} \geq 2$, for all $n \geq 1$). Indeed, since $r \geq c_3(N)$ by assumption (recall the definition of $c_3(N)$ in \eqref{c_3}), this follows from the fact that
\begin{equation}\label{L:CRAMproof4bis}
\begin{split}
4 \Big[\frac{L_{n}}{r}\Big] \cdot N4^d &< 4 \Big[\frac{L_{n}}{r}\Big] \cdot N4^d + 2L_{n-1} \\
&\leq 4 \frac{L_{n}}{r} \cdot N4^d + \frac{1}{r}L_n = \frac{4 N4^d + 1}{r}L_n \stackrel{\eqref{c_3}}{<} \frac{c_3(N)}{r}L_n \leq L_n
\end{split}
\end{equation}
(the reason for adding $2L_{n-1}$ will become apparent in a moment). Hence, there exist $N4^d+1$ distinct vertices $z_j \in \mathbb{L}_{n-1}\cap S(x_0, L_n+ 4[L_{n}/r]j)$, $0 \leq j \leq N 4^d$, such that $\pi \cap B(z_j,L_{n-1})\ne \emptyset$, for all $j$. Moreover, \eqref{L:CRAMproof4bis} also shows that  
\begin{equation}\label{L:CRAMproof4bisbis}
B(z_j,2L_{n-1}) \subset \big( \mathbb{Z}^d \cap(x_0 + [-2L_n,2L_n)^d)\big), \text{ for all $0 \leq j \leq N4^d$}, 
\end{equation}
thus $\pi$ connects $ B(z_j,L_{n-1})$ to $ S(z_j,2L_{n-1})$ (since it connects $B(x_0,L_n)$ to $S(x_0,2L_n)$), i.e. the events $\big\{B(z_j,L_{n-1})\stackrel{\text{bad}}{\longleftrightarrow}S(z_j,2L_{n-1})\big\}$, $0 \leq j \leq N 4^d$, all occur. Note that, since $2L_{n-1} \leq [L_n/r]$, the sets $B(z_j,2L_{n-1})$, $0 \leq j \leq N 4^d$, are all disjoint. 

We now apply the induction hypothesis individually to each of the events $\big\{B(z_j,L_{n-1})\stackrel{\text{bad}}{\longleftrightarrow}S(z_j,2L_{n-1})\big\}$, $0 \leq j \leq N 4^d$, and obtain that
\begin{equation}\label{L:CRAMproof1_2}
\begin{split}
&\text{there exist $N4^d+1$ vertices $y_j \in \mathbb{L}_{n-1} \cap (z_j + [-2L_{n-1}, 2L_{n-1})^d)$,} \\ 
&\text{$0 \leq j \leq N 4^d$, which are $(n-1)$-bad.}
\end{split}
\end{equation}
Furthermore, by construction, for arbitrary $0 \leq j<j' \leq N 4^d$,
\begin{equation} \label{L:CRAMproof1_3}
\begin{split}
|y_j-y_{j'}|_\infty &\geq d\big(B(z_j, 2L_{n-1}), B(z_{j'}, 2L_{n-1})\big)\\
&=|z_j-z_{j'}|_\infty - 4L_{n-1} \geq 4 \Big[\frac{L_{n}}{r}\Big] - 4L_{n-1} > \Big[\frac{L_{n}}{r}\Big],
\end{split}
\end{equation}
where we have used in the last step that $2L_{n-1} \leq [L_n/r]$, for all $n\geq 1$ (see above \eqref{L:CRAMproof4bis}).    

The family $\{B_{n,y'}; \; y' \in \mathbb{L}_n \cap \big( x_0 + [-2L_n,2L_n)^d \big)\}$ consists of $4^d$ disjoint boxes (recall \eqref{B_n,x}), the union of which is the set $\mathbb{Z}^d \cap \big( x_0 + [-2L_n,2L_n)^d \big)$. In particular, by \eqref{L:CRAMproof1_2} and \eqref{L:CRAMproof4bisbis}, the latter set contains all $N4^d+1$ vertices $y_j$, $0\leq j \leq N4^d$, hence,
\begin{equation}\label{L:CRAMproof1_4}
\begin{split}
&\text{there exists a vertex $y  \in \mathbb{L}_n \cap \big( x_0 + [-2L_n,2L_n)^d \big)$ such that $B_{n,y}$} \\ 
&\text{contains at least $N+1$ of the vertices in the set $\{y_j; \; 0\leq j \leq N4^d \}$}.
\end{split}
\end{equation}
We now show that 
\begin{equation} \label{L:CRAMproof1_5}
\text{the vertex $y$ defined by \eqref{L:CRAMproof1_4} is $n$-bad}.
\end{equation}
By \eqref{L:CRAMproof1_2} and \eqref{L:CRAMproof1_4}, the set $B_{n,y} \cap \{y_j; \; 0\leq j \leq N4^d \}$ contains at least $N+1$ distinct $(n-1)$-bad vertices, each having one (or several) of only $N$ different types (see below \eqref{bad_events_general}). Hence, there are at least two vertices $\bar{y}_1,\bar{y}_2 \in B_{n,y} \cap \{y_j; \; 0\leq j \leq N4^d \}$ which have the same type, i.e. for some $i_0 \in \{1,\dots, N\}$, the events $\mathcal{B}_{n-1,\bar{y}_1}^{(i_0)}$ and $\mathcal{B}_{n-1,\bar{y}_2}^{(i_0)}$ both occur. By \eqref{L:CRAMproof1_3}, $|\bar{y}_1-\bar{y}_2|_\infty > L_n/r$. Thus, the event
\begin{equation*} 
\mathcal{B}_{n,y}^{(i_0)} \stackrel{\eqref{A_n,xRECURSION}}{=} \bigcup_{\substack{y',\; y'' \in \mathbb{L}_{n-1}\cap B_{n,y} \\ |y'-y''|_\infty > L_{n}/r}} \mathcal{B}_{n-1,y'}^{(i_0)} \cap \mathcal{B}_{n-1,y''}^{(i_0)} \quad \big(\supseteq \mathcal{B}_{n-1,\bar{y}_1}^{(i_0)} \cap \mathcal{B}_{n-1,\bar{y}_2}^{(i_0)}\big)
\end{equation*}
occurs, that is, the vertex $y$ is $n$-bad of type $i_0$, and in particular, it is $n$-bad. Therefore, \eqref{L:CRAMproof1_5} holds, which yields \eqref{L:CRAMproof1_1}, and completes the proof of Lemma \ref{L:CRAMGEN}.
\end{proof}

The geometric Lemma \ref{L:CRAMGEN} has the following, more quantitative corollary, tailored to our future purposes.

\begin{lemma} \label{L:CRAM} $(N \geq 1)$ 

\medskip
\noindent If, for some $L_0 \geq 1$, $r \geq c_3(N)$, and $l_0 \geq 2r$,
\begin{equation} \label{cram_condition}
 P\big[\mathcal{B}_{n,x}^{(i)}\big] \leq 2^{-2^n},  \text{ for all $n\geq 0$, $x\in \mathbb{L}_n$, $1 \leq i \leq N$},
\end{equation}
then, for this choice of $L_0$, $r$ and $l_0$,
\begin{equation} \label{cram_conclusion}
 P\big[B(x,L_n)\stackrel{\text{bad}}{\longleftrightarrow}S(x,2L_n) \big] \leq N4^d \cdot  2^{-2^n}, \text{ for all $n\geq 0$, $x\in \mathbb{L}_n$.}
\end{equation}
\end{lemma}

In essence, Lemma \ref{L:CRAM} asserts that the probability of having long paths of bad vertices in $\mathbb{L}_0$ is small, provided \eqref{cram_condition} holds. In applications, we will use Proposition \ref{P:DEC_INEQ_PROPAGATED} to ensure the latter condition is satisfied.

\begin{proof}
The conditions of Lemma \ref{L:CRAMGEN} are satisfied, thus \eqref{L:CRAMGEN_conclusion} holds. Hence,
\begin{equation*}
 P\big[B(x,L_n)\stackrel{\text{bad}}{\longleftrightarrow}S(x,2L_n) \big] \stackrel{\eqref{L:CRAMGEN_conclusion}}{\leq} 4^d \sup_{y\in \mathbb{L}_n} P \Big[ \bigcup_{i=1}^N \mathcal{B}_{n,y}^{(i)} \Big] \stackrel{\eqref{cram_condition}}{\leq} N4^d \cdot 2^{-2^n},
\end{equation*}
for all $n \geq 0$ and $x\in \mathbb{L}_n$, which yields \eqref{cram_conclusion}.
\end{proof}

\begin{remark} 
 1) Even though we will not need this below, we note that the conclusion \eqref{L:CRAMGEN_conclusion} of Lemma \ref{L:CRAMGEN} (and \eqref{cram_conclusion} of Lemma \ref{L:CRAM}) continues to hold if one replaces $\big\{B(x,L_n)\stackrel{\text{bad}}{\longleftrightarrow}S(x,2L_n)\big\}$ in \eqref{L:CRAMGEN_conclusion} by the event that $B(x,L_n)$ is connected to $S(x,2L_n)$ by a $*$-nearest-neighbor path of bad vertices in $\mathbb{L}_0$. The above proof still works in this case. Moreover, for the sole purpose of Lemma \ref{L:CRAMGEN}, the only relevant condition on the events $\mathcal{B}_{n,x}^{(i)}$, apart from the ``cascading property''  \eqref{A_n,x}, is that the events at level $n=0$ be $\mathbb{L}_0$-adapted. But their precise form, as given by \eqref{bad_events_general}, is immaterial (only adapted to our later purposes). 

\medskip
\noindent 2) The bound \eqref{cram_conclusion} obtained in Lemma \ref{L:CRAM} will enable us to deduce not only that the connectivity functions of the random sets $\mathcal{L}^{\geq h}$, $\mathcal{I}^{u,\alpha}$, $\mathcal{V}^{u,\alpha}$ tend to $0$ as distance grows to infinity in a certain region of their respective parameter spaces (i.e. that there is a non-trivial sub-critical regime), but also that their decay is stretched exponential when the parameters are sufficiently ``far away'' from the critical points, see \eqref{T:GFF_NOPERC3} and Remark \ref{R:U_*} below. \hfill $\square$
\end{remark}

\medskip

With Lemma \ref{L:CRAM} at hand, we are ready to prove the main result of this section concerning ``two-sided'' level set percolation for the Gaussian free field on $\mathbb{Z}^d$, $d\geq 3$. We introduce the map $\psi_h: \mathbb{R}^{\mathbb{Z}^d}\rightarrow \{ 0,1 \}^{\mathbb{Z}^d}$, $(\varphi_x)_{x \in \mathbb{Z}^d} \mapsto \big( 1\{ |\varphi_x| \geq h \} \big)_{x\in \mathbb{Z}^d}$, for $h\geq 0$, consider the measure 
\begin{equation} \label{Q_h,G}
Q_h^G = \psi_h \circ P^G \text{ (on $\{0,1\}^{\mathbb{Z}^d}$)},
\end{equation}
and recall the definition of the critical parameter $\mathfrak{h}_* = \inf\{h \geq 0 ; \; Q_h^G[0 \leftrightarrow \infty] = 0\}$ in \eqref{h_*}. The following theorem strengthens the result $(2.65)$ of \cite{RS}.  

\begin{theorem} \label{T:GFF_NOPERC} For all sufficiently large levels $h \geq 0$, the sets $\mathcal{L}^{\geq h}$ do not percolate, i.e. 
\begin{equation} \label{T:GFF_NOPERC1}
(0\leq) \ \mathfrak{h}_*(d) < \infty,  \text{ for all $d\geq 3$}.
\end{equation}
Moreover, there exist positive constants $c_4$, $c$, $c'$ and $0 < \rho <1$ such that
\begin{equation}  \label{T:GFF_NOPERC2}
Q_h^G [B(0,L) \longleftrightarrow S(0,2L)] \leq c \cdot e^{-c'L^{\rho}},  \text{ for all $h \geq c_4$ and $L \geq 1$}. 
\end{equation}
\end{theorem}

\begin{proof} First, we note that \eqref{T:GFF_NOPERC1} follows from \eqref{T:GFF_NOPERC2}. Indeed,
\begin{equation} \label{T:GFF_NOPERCpf1}
 Q_h^G[0 \longleftrightarrow \infty] \leq Q_h^G [B(0,L) \longleftrightarrow S(0,2L)],  \text{ for all } h\geq0, \; L \geq 1.
\end{equation}
By \eqref{T:GFF_NOPERC2}, the quantity on the right-hand-side goes to zero as $L \to \infty$ for $h\geq c_4$, whence $Q_h^G[0 \longleftrightarrow \infty]=0$ for all $h\geq c_4$, i.e. $\mathfrak{h}_* \leq c_4 \ (< \infty)$.

We now turn to the proof of \eqref{T:GFF_NOPERC2}, which makes use of the renormalization scheme introduced above. Thus, we consider a sequence of length scales $(L_{n})_{n\geq 0}$ as defined in \eqref{L_n}, with
\begin{equation} \label{T:GFF_NOPERCpf2}
 L_0 = 10, \quad l_0 = 2r, \quad r=c_3(2) \quad \text{(see \eqref{c_3} for the definition of $c_3(N)$)},
\end{equation}
and corresponding renormalized lattices $(\mathbb{L}_n)_{n\geq0}$, see \eqref{LL_n}. In particular, condition \eqref{rparameter} is satisfied, and $l_0 \geq 100$ by definition of $c_3(N)$. For $h\geq 0$ and $x\in \mathbb{L}_0$, we introduce the events
\begin{equation} \label{T:GFF_NOPERCpf3}
 \begin{split}
  &\mathcal{B}_{0,x}^{(1)}(h) = \big\{ \max_{y\in B_{0,x}} \varphi_y \geq h \big\}, \\ 
  &\mathcal{B}_{0,x}^{(2)}(-h) = \big\{ \min_{y\in B_{0,x}} \varphi_y \leq - h \big\}
 \end{split}
\end{equation}
(see \eqref{B_n,x} for the definition of the boxes $B_{n,x}$), and  define the events $\mathcal{B}_{n,x}^{(1)}(h)$, $\mathcal{B}_{n,x}^{(2)}(-h)$, for all $n\geq 0$, $x\in \mathbb{L}_n$, $h\geq 0$, as in \eqref{A_n,x}, i.e.
\begin{equation} \label{T:GFF_NOPERCpf4} 
\mathcal{B}_{n,x}^{(1)}(h)  =  \bigcup_{\mathcal{T}\in \Lambda_{n,x}} \bigcap_{m\in T^{(n)}} \mathcal{B}_{0,\mathcal{T}(m)}^{(1)}(h), 
\end{equation}
and similarly for $\mathcal{B}_{n,x}^{(2)}(-h)$. Since $\mathcal{B}_{0,x}^{(1)}(h)= \Phi_h^{-1}(A_{0,x})$ with $\Phi_h$ given by \eqref{phi_h} and $A_{0,x}$ as defined in \eqref{A}, we obtain
\begin{equation} \label{T:GFF_NOPERCpf5}
P^G[\mathcal{B}_{n,x}^{(1)}(h)] \stackrel{\eqref{T:GFF_NOPERCpf4}}{\leq} |\Lambda_{n,x}| \cdot \sup_{\mathcal{T}\in \Lambda_{n,x}}P^G\Big[\bigcap_{m\in T^{(n)}} \mathcal{B}_{0,\mathcal{T}(m)}^{(1)}(h) \Big] \stackrel{ \eqref{Lambda_n,x}, \eqref{p_n}}{\leq} (l_0^{2d})^{2^n} \cdot p_n^A(h), 
\end{equation}
for $n\geq 0$, $x\in \mathbb{L}_n$, $h\geq 0$. For the remainder of this proof, $A$ will solely refer to the collection of events defined in \eqref{A}.

We now consider the sequence $(h_n)_{n\geq 0}$ defined in \eqref{h_n}. Since $L_0$, $r$ and $l_0$ are fixed, cf. \eqref{T:GFF_NOPERCpf2}, only $h_0$ in \eqref{h_n} remains to be selected. By an elementary estimate, we obtain, for all $h_0 > 0$,
\begin{equation} \label{T:GFF_NOPERCpf6}
\begin{split}
 p_0^A(h_0) \stackrel{\eqref{p_n},\eqref{T:GFF_NOPERCpf3}}{=} P^G \big[\mathcal{B}_{0,0}^{(1)}(h_0)\big] 
& = P^G \big[\max_{y\in B_{0,0}} \varphi_y \geq h_0\big] \\
&\leq  |B_{0,0}| \cdot (2 \pi g(0))^{-1/2} \int_{h_0}^\infty e^{-t^2/2g(0)} dt \\
&\leq c h_0^{-1}e^{-h_0^2/2g(0)},
\end{split}
\end{equation}
where we have also used translation invariance, see Remark \ref{R:DEC_INEQ}, 2), in the first step. (One could also have used the BTIS-inequality, see for example Theorem 2.1.1 in \cite{AT}, to bound $p_0^A(h_0)$.) Now, since $K_0$ as defined in \eqref{K_0} is completely determined by the choice of $l_0$ in \eqref{T:GFF_NOPERCpf2}, \eqref{T:GFF_NOPERCpf6} yields
\begin{equation} \label{T:GFF_NOPERCpf8}
 p_0^A(h_0) \leq e^{-K_0}, \text{ for all } h_0 \geq c',
\end{equation}
i.e., condition \eqref{p_0_cond} holds for sufficiently large $h_0$. Setting $h_0=c'$ and recalling that $h_\infty = \lim_{n\to\infty} h_n$ is finite, see \eqref{h_infty}, we thus obtain, by virtue of Proposition \ref{P:DEC_INEQ_PROPAGATED},
\begin{equation} \label{T:GFF_NOPERCpf9}
 p_n^A(h_\infty) \stackrel{\eqref{p_n_monotone}}{\leq} p_n^A(h_n) \stackrel{\eqref{p_n_bounds}, \; \eqref{K_0}}{\leq} (2  l_0^{2d})^{-2^n}, \text{ for all } n\geq 0,
\end{equation}
and thus, together with \eqref{T:GFF_NOPERCpf5},
\begin{equation} \label{T:GFF_NOPERCpf10}
 P^G \big[\mathcal{B}_{n,x}^{(1)}(h_\infty)\big] \leq 2^{-2^n}, \text{ for all $n\geq 0$, $x\in \mathbb{L}_n$}.
\end{equation}
Since $\varphi = (\varphi_x)_{x\in \mathbb{Z}^d}$ has the same law as $-\varphi$ under $P^G$, we deduce from the definition \eqref{T:GFF_NOPERCpf3} that $P^G\big[\mathcal{B}_{n,x}^{(2)}(-h)\big]= P^G\big[\mathcal{B}_{n,x}^{(1)}(h)\big]$, for all $h\geq0$, $n\geq 0$ and $x\in \mathbb{L}_n$. Thus, the bounds \eqref{T:GFF_NOPERCpf10} also hold with $\mathcal{B}_{n,x}^{(2)}(-h_\infty)$ in place of $\mathcal{B}_{n,x}^{(1)}(h_\infty)$.

We define a vertex $y\in \mathbb{L}_0$ to be bad if the event $\mathcal{B}_{0,y}^{(1)}(h_\infty) \cup \mathcal{B}_{0,y}^{(2)}(-h_\infty)$ occurs, i.e. if $|\varphi_z| \geq h_{\infty}$ for some $z\in B_{0,y}$, and observe that, if there exists a nearest-neighbor path $\pi$ in $\mathcal{L}^{\geq h_\infty}$ connecting $B(x,L_n)$ to $S(x,2L_n)$, for some $n \geq 0$ and $x \in \mathbb{L}_n$, then there exists a nearest-neighbor path of bad vertices in $\mathbb{L}_0$ connecting $B(x,L_n)\cap \mathbb{L}_0$ to $S(x,2L_n)\cap \mathbb{L}_0$. Indeed, the trace $\pi_0$ of $\pi$ on $\mathbb{L}_0$, cf. \eqref{trace}, is such a path: by construction, it is a nearest-neighbor path in $\mathbb{L}_0$, see \eqref{trace_path_nn}, which connects $B(x,L_n)\cap \mathbb{L}_0$ to $S(x,2L_n)\cap \mathbb{L}_0$. Moreover, if $y$ is any vertex in $\mathbb{L}_0$ traversed by $\pi_0$, then $\pi \cap B_{0,y} \ne \emptyset$, so in particular $\max_{z\in B_{0,y}} |\varphi_z| \geq h_\infty$, i.e. $y$ is a bad vertex. The probability of the event $\big\{ B(x,L_n)\stackrel{\text{bad}}{\longleftrightarrow}S(x,2L_n) \big\}$, which refers to the existence of a nearest-neighbor path of bad vertices (in $\mathbb{L}_0$) connecting $B(x,L_n)\cap \mathbb{L}_0$ to $S(x,2L_n)\cap \mathbb{L}_0$, can be bounded using Lemma \ref{L:CRAM} (with $N=2$, $P=P^G$, $\zeta_x^{(1)}= 1\{ \varphi_x \geq h_\infty\}$ and $\zeta_x^{(2)}= 1\{ \varphi_x \leq - h_\infty\}$, $x \in \mathbb{Z}^d$, see \eqref{bad_events_general} and \eqref{T:GFF_NOPERCpf3}), which applies due to \eqref{T:GFF_NOPERCpf10} and the choices in \eqref{T:GFF_NOPERCpf2}. All in all, we obtain
\begin{equation} \label{T:GFF_NOPERCpf11}
Q_{h_\infty}^G[B(x,L_n) \longleftrightarrow S(x,2L_n)] \leq P^G\big[B(x,L_n)\stackrel{\text{bad}}{\longleftrightarrow}S(x,2L_n) \big] \stackrel{\eqref{cram_conclusion}}{\leq} c 2^{-2^n}, 
\end{equation}
for all $n\geq 0$, $x \in \mathbb{L}_n$. We now set $\rho = \log 2 / \log l_0$, whence $2^n = l_0^{n\rho} = (L_n / L_0)^{\rho}$. Given $L \geq 1$, we first assume there exists $n \geq 0$ such that $4L_n \leq L < 4L_{n+1}$. Then, since 
\begin{equation} \label{T:GFF_NOPERCpf12}
Q_{h_\infty}^G \big[ B(0,L) \longleftrightarrow  S(0,2L) \big] \leq Q_{h_\infty}^G \bigg[ \bigcup_{ \substack{x \in \mathbb{L}_n : \\ B(x,L_n) \cap  S(0,L) \neq \emptyset}} \big\{ B(x,L_n) \longleftrightarrow S(x,2L_n) \big\} \bigg],
\end{equation}
and the number of sets contributing to the union on the right-hand side is bounded by $cl_0^{d-1}$, \eqref{T:GFF_NOPERCpf11} and \eqref{T:GFF_NOPERCpf12} readily imply \eqref{T:GFF_NOPERC2} with $h=h_\infty$, for all  $L\geq 4L_0$, and by adjusting the constants $c,c'$, for $L < 4 L_0$ as well. Observing that $Q_{h}^G [ B(0,L) \longleftrightarrow  S(0,2L) ]$ is a decreasing function of $h$, for all $L\geq 1$, it follows that \eqref{T:GFF_NOPERC2} holds for all $h\geq h_\infty \stackrel{\text{def.}}{=}c_4$. This completes the proof of Theorem \ref{T:GFF_NOPERC}.
\end{proof}

\begin{remark} \label{R:h_*positive} 
1) From  \eqref{T:GFF_NOPERC2}, one easily deduces that
\begin{equation} \label{T:GFF_NOPERC3}
Q_h^G [0 \longleftrightarrow x] \leq c \cdot e^{-c''|x|^{\rho}}, \text{ for all $h \geq c_4$ and $x \in \mathbb{Z}^d$}, 
\end{equation}
for some positive constant $c''$ and $c, \rho $ as in \eqref{T:GFF_NOPERC2}, i.e. that the connectivity function $Q_h^G [0 \longleftrightarrow x]$ of the level set $\mathcal{L}^{\geq h}$ has stretched exponential decay in $x$ for sufficiently large $h$. Indeed, \eqref{T:GFF_NOPERC3} follows since
\begin{equation*}
Q_{h}^G [ 0 \longleftrightarrow  x ] \leq Q_{h}^G [ B(0,L) \longleftrightarrow  S(0,2L) \big] \stackrel{\eqref{T:GFF_NOPERC2}}{\leq} c \cdot e^{- c''|x|^\rho}
\end{equation*}
whenever $ 2L \leq |x|_\infty < 2(L+1)$, for all $h \geq c_4$.

\medskip
\noindent 2) $\mathfrak{h}_*$ is \textit{strictly} positive in large dimensions. Let $h_*$ denote the critical parameter for percolation of the sets $E_\varphi^{\geq h} = \{x\in \mathbb{Z}^d ; \; \varphi_x \geq h\}$, for $h \in \mathbb{R}$. Since $E_\varphi^{\geq h}$ is contained in $\mathcal{L}^{\geq h}$ for all $h\geq0$, one has $h_{*}(d)\leq \mathfrak{h}_{*}(d)$ for all $d\geq 3$, and by Theorem 3.3 of \cite{RS}, $h_*$ is positive in large dimensions. \hfill $\square$
\end{remark}

\section{Absence of percolation for $\mathcal{I}^{u,\alpha}$ and $\mathcal{V}^{u,\alpha}$} \label{RI_NOPERC}

We now return to random interlacements, and consider the sets $\mathcal{I}^{u,\alpha}$, $\mathcal{V}^{u,\alpha}$ defined in  \eqref{I^u,alpha}, for $u,\alpha \geq 0$. We deduce in Theorem \ref{T:ALPHA_*} below, using the isomorphism theorem \eqref{isom_THM} in conjunction with Theorem \ref{T:GFF_NOPERC}, that $\mathcal{I}^{u,\alpha}$ contains (almost surely) no infinite cluster, for arbitrary $u\geq 0$ and $\alpha = \alpha(u)$ sufficiently large (recall the discussion following \eqref{isom_THM}, which explains the intuitive idea). By similar methods, we are able to show that $\mathcal{V}^{u,\alpha}$, for $ \alpha \geq 0$ and large enough $u = u(\alpha)$, does not percolate either, see Theorem \nolinebreak \ref{T:U_*}. 

Without further ado, we recall the definition of the critical parameter $\alpha_*(u)$ in \eqref{alpha_*} and begin by collecting a few elementary properties of this function. Clearly, $\alpha_*(0)=0$. By construction, see \eqref{L} and \eqref{Q_u,alpha}, the measure $Q_{u', \alpha}$ stochastically dominates the measure $Q_{u, \alpha}$, for all $u' > u \geq 0$ and $\alpha \geq 0$. Noting that $\alpha_*(u) = \sup \{ \alpha \geq 0 \; ; \; Q_{u,\alpha}[0\longleftrightarrow \infty] > 0 \}$ (with the convention $\sup \emptyset = 0$), it follows that 
\begin{equation} \label{alpha_*_monotone}
 \text{$\alpha_*(u)$ is a non-decreasing function of $u \geq 0$}.
\end{equation}
Furthermore, we have the following
 
\begin{theorem} \label{T:ALPHA_*} For all $u \geq 0$,
\begin{equation} \label{T:ALPHA_*1}
\alpha_{*}(u) < \infty
\end{equation} 
$($we recall that $\alpha_*(u)>0$, when $u>0$, by \eqref{T:U_*perc1}$)$. Moreover, for all $u \geq 0$, there exist positive constants $c_5(u)$, $c$, $c'$, and $0 < \rho <1$ such that
\begin{equation} 
 \mathbb{P} \big[B(0,L) \stackrel{\mathcal{I}^{u,\alpha}}{\longleftrightarrow} S(0,2L) \big] \leq c \cdot e^{-c'L^{\rho}},  \text{ for all $\alpha \geq c_5(u)$ and $L \geq 1$}  \label{T:ALPHA_*2}.
\end{equation}
\end{theorem}
\begin{proof}
Let $u \geq 0$. The finiteness of $\alpha_*(u)$ in \eqref{T:ALPHA_*1} follows from \eqref{T:ALPHA_*2}. Indeed, this follows from the same argument as the one used to deduce \eqref{T:GFF_NOPERC1} from \eqref{T:GFF_NOPERC2}. 

We now prove \eqref{T:ALPHA_*2}. Clearly, the law $Q_{u,\alpha}$ of $\big(1\{L_{x,u} > \alpha \} \big)_{x \in \mathbb{Z}^d}$ under $\mathbb{P}$, is stochastically dominated by the law of $\big(1\{L_{x,u} + \varphi_x^2/2 > \alpha \} \big)_{x \in \mathbb{Z}^d}$ under $\mathbb{P} \otimes P^G$, for any $u, \alpha \geq 0$. By the Isomorphism Theorem \eqref{isom_THM}, the latter is the same as the law of $\big(1\{(\varphi_x + \sqrt{2u})^2/2 > \alpha \} \big)_{x \in \mathbb{Z}^d}$, under $P^G$. Moreover, $(\varphi_x + \sqrt{2u})^2/2 > \alpha$ implies $|\varphi_x| > \sqrt{2\alpha} - \sqrt{2u}.$ Thus, for all $u \geq 0$ and $\alpha \geq u$,
\begin{equation}\label{stoch_dom}
\text{$Q_{u,\alpha}$ is stochastically dominated by $Q^G_{h(u,\alpha)}$, with $h(u,\alpha)= \sqrt{2\alpha} - \sqrt{2u} \ (\geq 0)$}
\end{equation}
(see \eqref{Q_h,G} for the definition of $Q^G_h$). In particular, since $\{B(0,L) \longleftrightarrow S(0,2L)\}$ is an increasing event, we obtain, for any $\alpha \geq (\sqrt{2u}+ c_4)^2/2$ (see Theorem \ref{T:GFF_NOPERC} for the definition of $c_4$),
\begin{equation} \label{T:ALPHA*_pf1}
Q_{u,\alpha}[B(0,L) \longleftrightarrow S(0,2L)] \stackrel{\text{\eqref{stoch_dom}}}{\leq}Q_{h(u,\alpha)}^G[B(0,L) \longleftrightarrow S(0,2L)] \stackrel{\eqref{T:GFF_NOPERC2}}{\leq} c \cdot e^{c'L^{\rho}},  \text{ for $L \geq 1$},
\end{equation}
for some positive constants $c,c'$ and $0 < \rho < 1$, where we have used in the last step that $\alpha \geq (\sqrt{2u}+ c_4)^2/2$ implies $h(u,\alpha) \geq c_4$. Hence, \eqref{T:ALPHA*_pf1} yields \eqref{T:ALPHA_*2} with $c_5(u)= (\sqrt{2u}+ c_4)^2/2$. This completes the proof of Theorem \ref{T:ALPHA_*}.
\end{proof}

Finally, we consider the set $\mathcal{V}^{u,\alpha}$, with $u, \alpha \geq0$, and recall the definition of the critical parameter $u_*(\alpha)$ in \eqref{u_*}. We observe that
\begin{equation} \label{u_*_monotone}
 \text{$u_*(\alpha)$ is a non-decreasing function of $\alpha \geq 0$}.
\end{equation}
Indeed, $\mathcal{V}^{u,\alpha'}(\omega) \supset \mathcal{V}^{u,\alpha}(\omega)$, for all $u \geq 0$, $\alpha' > \alpha \geq 0$ and $\omega \in \Omega$ (the space on which $\mathbb{P}$ is defined, see below \eqref{intensity}), thus $ \mathbb{P}\big[ 0 \stackrel{\mathcal{V}^{u,\alpha'}}{\longleftrightarrow}\infty \big] \geq\mathbb{P}\big[ 0 \stackrel{\mathcal{V}^{u,\alpha}}{\longleftrightarrow}\infty \big]$, which yields \eqref{u_*_monotone} since $u_*(\alpha) = \sup \big\{ u\geq 0 ; \; \mathbb{P}\big[ 0 \stackrel{\mathcal{V}^{u,\alpha}}{\longleftrightarrow}\infty \big] > 0 \big\}$.

We also note that, since $\mathcal{V}^{u,0} = \{ x \in \mathbb{Z}^d ; \; L_{x,u} = 0\}$ coincides with the vacant set $\mathcal{V}^u$ of random interlacement at level $u$ introduced in \cite{S1}, we have $u_*(0)=u_*$, where $u_*$ refers to the critical parameter for percolation of $\mathcal{V}^u$. As mentioned in the introduction (see the references below \eqref{u_*}), $u_*$ is known to be strictly positive (and finite) for all $d \geq 3$. We now prove that the set $\mathcal{V}^{u,\alpha}$ undergoes a non-trivial percolation phase transition as $u\geq0$ varies, for every (fixed) value of $\alpha \geq 0$.

\begin{theorem} \label{T:U_*} For all $\alpha \geq 0$,
\begin{equation} \label{T:U_*1}
 (0 <) \ u_*=  u_*(0) \leq u_*(\alpha) < \infty .
\end{equation}
Moreover, for all $\alpha \geq 0$, there exist positive constants $c_6(\alpha)$, $c$ , $c'$ and $0< \rho<1$ such that for all $ u\geq c_6(\alpha)$ and all $L \geq 1$,
\begin{equation} \label{T:U_*2}
 \mathbb{P}\big[B(0,L) \stackrel{\mathcal{V}^{u,\alpha}}{\longleftrightarrow} S(0,2L) \big] \leq c e^{-c' L^\rho}. 
\end{equation}
\end{theorem}

\begin{proof}
We begin with \eqref{T:U_*1}. The inequality $u_*(\alpha) \geq u_*$ for all $\alpha \geq 0$ is immediate from \eqref{u_*_monotone} and $u_*(0) = u_*$. The finiteness of $u_*(\alpha)$ follows from \eqref{T:U_*2} (just as \eqref{T:ALPHA_*2} implies \eqref{T:ALPHA_*1}, see the proof of Theorem \ref{T:ALPHA_*}).

It thus remains to show \eqref{T:U_*2}. The proof encompasses a renormalization argument. Thus, we consider the increasing sequence of length scales $(L_n)_{n\geq 0}$ defined in \eqref{L_n} (and corresponding lattices $\mathbb{L}_n$, for $n \geq 0$, see \eqref{LL_n}), with
\begin{equation} \label{T:U_*pf1}
 L_0 = 10, \ r = c_3(3) \text{ and } l_0=2r \quad \text{(see \eqref{c_3} for the definition of $c_3(N)$)}.
\end{equation}
In what follows, we identify any event $A$ occurring under $\mathbb{P}$ with the event $A \times \mathbb{R}^{\mathbb{Z}^d}$ occurring under $\mathbb{P} \otimes P^G$, and similarly any event $B$ under $P^G$ with $\Omega \times B$. For $\alpha,u \geq 0$ and $x\in \mathbb{L}_0$, we define the events (under $\mathbb{P} \otimes P^G$)
\begin{equation} \label{T:U_*pf2}
 \begin{split}
  &\mathcal{B}_{0,x}^{(1)}(u) = \big\{ \max_{y\in B_{0,x}} \varphi_y \geq \sqrt{u/2} \big\}, \\ 
  &\mathcal{B}_{0,x}^{(2)}(u) = \big\{ \min_{y\in B_{0,x}} \varphi_y \leq - \sqrt{u/2} \big\} \\
  &\mathcal{B}_{0,x}^{(3)}(\alpha,u) = \big\{ \min_{y\in B_{0,x}} \big(L_{y,u} + \frac{1}{2} \varphi_y^2 \big) <  \frac{1}{2} \big( \sqrt{2 \alpha} + \sqrt{u/2} \big)^2 \big\}
 \end{split}
\end{equation}
(see \eqref{B_n,x} for the definition of the boxes $B_{0,x}$). We call a vertex $x\in \mathbb{L}_0$ $(\alpha,u)$-bad if 
\begin{equation} \label{u,alpha_bad_events}
\mathcal{B}_{0,x}^{(1)}(u) \cup \mathcal{B}_{0,x}^{(2)}(u) \cup \mathcal{B}_{0,x}^{(3)}(\alpha, u)
\end{equation}
occurs under $\mathbb{P} \otimes P^G$. The reason for the choices in \eqref{T:U_*pf2} is the following

\begin{lemma} \label{L:U,ALPHA_BAD} $(u,\alpha \geq 0)$

\medskip
\noindent For all $n \geq 0$ and $x \in \mathbb{L}_n$,
 \begin{equation} \label{L:U,ALPHA_BAD1}
  \mathbb{P}\big[ B(x,L_n) \stackrel{\mathcal{V}^{u,\alpha}}{\longleftrightarrow} S(x,2L_n)\big] \leq \mathbb{P} \otimes P^G \big[ B(x,L_n) \stackrel{(\alpha,u)\text{-bad}}{\longleftrightarrow} S(x,2L_n) \big],
 \end{equation}
where $\big\{ B(x,L_n) \stackrel{(\alpha,u)\text{-bad}}{\longleftrightarrow} S(x,2L_n) \big\}$ is the event that there exists a nearest-neighbor path of $(\alpha,u)$-bad vertices in $\mathbb{L}_0$ connecting $B(x,L_n)\cap \mathbb{L}_0$ to $S(x,2L_n)\cap \mathbb{L}_0$.
\end{lemma}
\begin{proof1}
By translation invariance, it suffices to consider the case $x=0$. Let $n\geq 0$. By definition, if $ \big\{B(0,L_n) \stackrel{\mathcal{V}^{u,\alpha}}{\longleftrightarrow} S(0,2L_n) \big\}$ occurs, there exists a nearest-neighbor path $\pi$ of vertices in $\mathcal{V}^{u,\alpha}$ connecting $B(0,L_n)$ to $S(0,2L_n)$. We consider $\pi_0$, the trace of $\pi$ on $\mathbb{L}_0$, see \eqref{trace}, and show that 
\begin{equation} \label{L:U,ALPHA_BADpf1}
\text{all vertices traversed by $\pi_0$ are $(\alpha,u)$-bad.}
\end{equation}
This implies \eqref{L:U,ALPHA_BAD1}, for by construction (cf. \eqref{trace_path_nn}), $\pi_0$ is a nearest-neighbor path in $\mathbb{L}_0$ connecting $B(0,L_n)\cap \mathbb{L}_0$ to $S(0,2L_n)\cap \mathbb{L}_0$, whence $\big\{B(0,L_n) \stackrel{\mathcal{V}^{u,\alpha}}{\longleftrightarrow} S(0,2L_n) \big\} \times \mathbb{R}^{\mathbb{Z}^d} \subseteq \big\{ B(0,L_n) \stackrel{(\alpha,u)\text{-bad}}{\longleftrightarrow} S(0,2L_n) \big\}$, which yields \eqref{L:U,ALPHA_BAD1}. 

Let $x\in \mathbb{L}_0$ be any vertex in $\text{range}(\pi_0)$. By definition of $\pi_0$, $B_{0,x} \cap \pi \ne \emptyset$. In particular,
\begin{equation}\label{L:U,ALPHA_BADpf2}
 L_{y_0} \leq \alpha, \text{ for some $y_0 \in B_{0,x}$}.
\end{equation}
We now assume $\mathcal{B}_{0,x}^{(3)}(\alpha,u)$ does not occur, i.e. $L_{y,u} +  \varphi_y^2/2 \geq  \big( \sqrt{2 \alpha} + \sqrt{u/2} \big)^2/2$, for all $y \in B_{0,x}$. Then,
\begin{equation*}
 \alpha + \frac{1}{2} \varphi_{y_0}^2 \stackrel{\eqref{L:U,ALPHA_BADpf2}}{\geq} L_{y_0} + \frac{1}{2} \varphi_{y_0}^2 \geq \frac{1}{2} \big( \sqrt{2 \alpha} + \sqrt{u/2} \big)^2 \geq \alpha + \frac{u}{4},
\end{equation*}
i.e. $\varphi_{y_0}^2 \geq u/2$, where we have used $(a+b)^2 \geq a^2 + b^2$ for all $a,b \geq 0$ in the last step. Thus, $\mathcal{B}_{0,x}^{(1)}(u)  \cup \mathcal{B}_{0,x}^{(2)}(u) = \{\max_{y\in B_{0,x}} |\varphi_y| \geq \sqrt{u/2}\}$ occurs, and therefore $x$ is $(\alpha,u)$-bad. This completes the proof of \eqref{L:U,ALPHA_BADpf1}, and thus of Lemma \ref{L:U,ALPHA_BAD}. \hfill $\square$
\end{proof1}
We now return to the proof of \eqref{T:U_*2}, and define the events $ \mathcal{B}_{n,x}^{(1)}(u)$, $\mathcal{B}_{n,x}^{(2)}(u)$, $\mathcal{B}_{n,x}^{(3)}(\alpha,u)$, for $n\geq 0$, $x \in\mathbb{L}_n$, and $u,\alpha \geq 0$, by
\begin{equation*}
\mathcal{B}_{n,x}^{(1)}(u)  =  \bigcup_{\mathcal{T}\in \Lambda_{n,x}} \bigcap_{m\in T^{(n)}} \mathcal{B}_{0,\mathcal{T}(m)}^{(1)}(u),
\end{equation*}
and similarly for $\mathcal{B}_{n,x}^{(2)}(u)$, $\mathcal{B}_{n,x}^{(3)}( \alpha,u)$. As in \eqref{T:GFF_NOPERCpf5}, when $i=1$ or $2$, we obtain, using \eqref{T:U_*pf2}, \eqref{Lambda_n,x} and the symmetry of $P^G$,
\begin{equation} \label{T:U_*pf3}
 \mathbb{P} \otimes P^G [\mathcal{B}_{n,x}^{(i)}(u)] \leq \big(l_0^{2d}\big)^{2^n} \cdot p_n^A \big( \sqrt{u/2} \; \big), \text{ for all $n\geq 0$, $x \in\mathbb{L}_n$, $u \geq 0$},
\end{equation}
where $A$ refers to the family of events (on $\{0,1 \}^{\mathbb{Z}^d}$) defined in \eqref{A} and $p_n^A(\cdot)$ is given by \eqref{p_n}. As for the third collection of events, we claim that
 \begin{equation} \label{T:U_*pf4}
 \mathbb{P} \otimes P^G [\mathcal{B}_{n,x}^{(3)}(\alpha,u)] \leq \big(l_0^{2d}\big)^{2^n} \cdot p_n^A \big(\sqrt{u/2}-\sqrt{2 \alpha} \; \big), \text{ for all $n\geq 0$, $x \in\mathbb{L}_n$, $u, \alpha \geq 0$}.
\end{equation}
Indeed, the Isomorphism Theorem \eqref{isom_THM} applied to the events $\mathcal{B}_{0,x}^{(3)}(\alpha,u)$ from \eqref{T:U_*pf2} yields 
\begin{equation*} 
\begin{array}{rcl}
 \mathbb{P} \otimes P^G [\mathcal{B}_{n,x}^{(3)}( \alpha, u)] \hspace{-1ex} & \leq & \hspace{-1ex} |\Lambda_{n,x}| \cdot \underset{\mathcal{T} \in \Lambda_{n,x}}{\text{sup}}  \mathbb{P} \otimes P^G \Big[\underset{m \in T^{(n)}}{\bigcap} \mathcal{B}_{0,\mathcal{T}(m)}^{(3)}(\alpha, u)\Big] \\
&\stackrel{ \eqref{isom_THM}}{=} & \hspace{-1ex}  |\Lambda_{n,x}| \cdot \underset{\mathcal{T} \in \Lambda_{n,x}}{\text{sup}} P^G \Big[\underset{m \in T^{(n)}}{\bigcap} \big\{ \underset{y\in B_{0, \mathcal{T}(m)}}{\text{min}} \big| \varphi_y + \sqrt{2u} \big| < \sqrt{2 \alpha} + \sqrt{u/2} \big\} \Big] \\
& \leq & \hspace{-1ex} |\Lambda_{n,x}| \cdot \underset{\mathcal{T} \in \Lambda_{n,x}}{\text{sup}}  P^G \Big[\underset{m \in T^{(n)}}{\bigcap} \big\{ \underset{y\in B_{0, \mathcal{T}(m)}}{\text{min}} \varphi_y  < \sqrt{2 \alpha} - \sqrt{u/2} \big\} \Big],
\end{array}
\end{equation*}
which, using \eqref{Lambda_n,x} to bound $|\Lambda_{n,x}|$ and symmetry of $P^G$, yields \eqref{T:U_*pf4}.

Since $K_0$ as defined in \eqref{K_0} is fully determined by the choice of $l_0$ in \eqref{T:U_*pf1}, the BTIS-inequality yields (cf. the argument leading to \eqref{T:GFF_NOPERCpf8}) $p_0^A(h) \leq e^{-K_0}$, for all $h\geq c$. The sequence $(h_n)_{n\geq 0}$ defined in \eqref{h_n} with $h_0=c$ (and $L_0$, $l_0$, $r$ given by \eqref{T:U_*pf1}) has a finite limit $h_\infty = \lim_{n\to\infty} h_n$, see \eqref{h_infty}, and Proposition \ref{P:DEC_INEQ_PROPAGATED} implies that $p_n^A(h_\infty) \leq (2 l_0^{2d})^{-2^n}$, for all $n\geq 0$, cf. \eqref{T:GFF_NOPERCpf9}. Together with \eqref{T:U_*pf3}, \eqref{T:U_*pf4}, and setting $\sqrt{u_\infty(\alpha)/2} = \sqrt{2 \alpha} + h_\infty$, this yields
\begin{equation}  \label{T:U_*pf5}
 \begin{split}
 & \mathbb{P} \otimes P^G [\mathcal{B}_{n,x}^{(1)}(u_\infty(\alpha))] =  \mathbb{P} \otimes P^G [\mathcal{B}_{n,x}^{(2)}(u_\infty(\alpha))]  \stackrel{\eqref{T:U_*pf3},\eqref{p_n_monotone} }{\leq}2^{-2^n} \\
 & \mathbb{P} \otimes P^G [\mathcal{B}_{n,x}^{(3)}(\alpha,u_\infty(\alpha))] \stackrel{\eqref{T:U_*pf4}}{\leq}2^{-2^n},
 \end{split}
\end{equation}
for all $n\geq 0$, $x\in \mathbb{L}_n$ and $\alpha \geq 0$. On account of \eqref{T:U_*pf5} and the choice of $r$ in \eqref{T:U_*pf1}, Lemma \ref{L:CRAM} applies (with $N=3$, $P = \mathbb{P} \otimes P^G$, $\zeta_x^{(1)}=1\{ \varphi_x \geq \sqrt{u_\infty(\alpha)/2}\}$, $\zeta_x^{(2)}=1\{ \varphi_x \leq -\sqrt{u_\infty(\alpha)/2}\}$ and $\zeta_x^{(3)}=1\{ L_{y,u_\infty(\alpha)} + \varphi_y^2/2  <  \big( \sqrt{2 \alpha} + \sqrt{u_\infty(\alpha)/2} \big)^2/2\}$, $x \in \mathbb{Z}^d$, see \eqref{bad_events_general} and \eqref{T:U_*pf2}), and we obtain, using Lemma \ref{L:U,ALPHA_BAD}, that for all $n\geq 0$, $x\in \mathbb{L}_n$, $\alpha \geq 0$ and $u \geq u_\infty (\alpha)$ 
\begin{equation} \label{T:U_*pf6}
\begin{array}{rcl}
 \mathbb{P}\big[ B(x,L_n) \stackrel{\mathcal{V}^{u,\alpha}}{\longleftrightarrow} S(x,2L_n)\big] \hspace{-1ex} & \leq & \hspace{-1ex} \mathbb{P}\big[ B(x,L_n) \stackrel{\mathcal{V}^{u_{\infty}(\alpha),\alpha}}{\longleftrightarrow} S(x,2L_n)\big] \\
& \stackrel{\eqref{L:U,ALPHA_BAD1}}{\leq} & \hspace{-1ex} \mathbb{P} \otimes P^G \big[ B(x,L_n) \stackrel{(\alpha,u_{\infty}(\alpha))\text{-bad}}{\longleftrightarrow} S(x,2L_n) \big] \\
& \stackrel{\eqref{u,alpha_bad_events}, \eqref{cram_conclusion}}{\leq} & \hspace{-1ex} c 2^{-2^n},
\end{array}
\end{equation}
(where we also used in the first line that the law $\widetilde{Q}_{u,\alpha}$ of $(1\{x \in \mathcal{V}^{u,\alpha}\})_{x \in \mathbb{Z}^d}$, is stochastically dominated by $\widetilde{Q}_{u',\alpha}$, for all $u \geq u'$). By the same interpolation argument as that in \eqref{T:GFF_NOPERCpf12}, we deduce from \eqref{T:U_*pf6}
that $\mathbb{P}\big[ B(0,L) \stackrel{\mathcal{V}^{u,\alpha}}{\longleftrightarrow} S(0,2L)\big] \leq ce^{-c' L^\rho} $, for some $c,c' >0$, $\rho = \log 2 / \log l_0$ and all $u \geq u_\infty (\alpha)$. Hence, \eqref{T:U_*2} holds with $c_6(\alpha)= u_\infty (\alpha)$. This completes the proof of Theorem \ref{T:U_*}.
\end{proof}

\begin{remark} \label{R:U_*} 1) From \eqref{T:ALPHA_*2} and \eqref{T:U_*2}, one easily deduces that the connectivity functions $\mathbb{P}\big[0 \stackrel{\mathcal{I}^{u,\alpha}}{\longleftrightarrow} x\big]$, $\mathbb{P}\big[0 \stackrel{\mathcal{V}^{u,\alpha}}{\longleftrightarrow} x\big]$ have stretched exponential decay in $x$ as $|x| \to \infty$, for all $u \geq 0$ and $\alpha \geq c_5(u)$, respectively for all $\alpha \geq 0$ and $u \geq c_6(\alpha)$ (cf. Remark \ref{R:h_*positive}, 1)).

\medskip

\noindent 2) One may introduce the auxiliary critical parameters (cf. Figure \ref{phasediagram} in Section \ref{INTRODUCTION})
\begin{equation*} \label{alpha_**}
u_{**}(\alpha) = \inf \big\{ u \geq 0  ; \; \mathbb{P}\big[B(0,L) \stackrel{\mathcal{V}^{u,\alpha}}{\longleftrightarrow} S(0,2L)\big]  \text{ has stretched exponential decay in $L$}\big\},
\end{equation*}
and similarly $\alpha_{**}(u)$, where the infimum is now over all $\alpha \geq 0$ and the event in the probability is replaced by $\big\{ B(0,L) \stackrel{\mathcal{I}^{u,\alpha}}\longleftrightarrow S(0,2L) \big\}$ (note in particular that $u_{**}(0)$ coincides with the quantity $u_{**}$ from \cite{SS}). Theorems \ref{T:ALPHA_*} and \ref{T:U_*} imply that $\alpha_*(u) \leq  \alpha_{**}(u) < \infty$ for all $u \geq 0$ and $u_*(\alpha) \leq u_{**}(\alpha) < \infty$ for all $\alpha \geq 0$. This raises the important question whether the parameters actually coincide. \hfill $\square$
\end{remark}

\bigskip

\noindent \textbf{Acknowledgement.} The author wishes to thank his advisor A.-S. Sznitman for suggesting the problem and for useful discussions, and the anonymous referee for valuable comments.

\end{document}